\documentclass{article}

\PassOptionsToPackage{numbers}{natbib}
\usepackage{amsmath}
\usepackage[preprint]{neurips_2026}
\usepackage[utf8]{inputenc} 
\usepackage[T1]{fontenc}    
\usepackage{hyperref}       
\usepackage{cleveref}
\usepackage{url}            
\usepackage{booktabs}       

\usepackage{makecell}
\usepackage{amsfonts}       
\usepackage{nicefrac}       
\usepackage{microtype}      
\usepackage{xcolor}         
\usepackage[pdftex]{graphicx}

\title{Time-Inhomogeneous Preconditioned Langevin Dynamics}

\author{%
  Alexander~Falk\thanks{corresponding author} \\
  Institute of Visual Computing\\
  Graz University of Technology\\
  Inffeldgasse 16, 8010 Graz\\
  \texttt{falk@tugraz.at} \\
  \And
  Laurenz~Nagler\\
  Institute of Visual Computing\\
  Graz University of Technology\\
  Inffeldgasse 16, 8010 Graz\\
  \texttt{lnagler@tugraz.at} \\
    \And
  Andreas~Habring\\
  Institute of Visual Computing\\
  Graz University of Technology\\
  Inffeldgasse 16, 8010 Graz\\
  \texttt{andreas.habring@tugraz.at} \\
  \And
  Thomas~Pock\\
  Institute of Visual Computing\\
  Graz University of Technology\\
  Inffeldgasse 16, 8010 Graz\\
  \texttt{thomas.pock@tugraz.at} \\
}

\usepackage{amsopn}
\DeclareMathOperator{\diag}{diag}

\usepackage{siunitx}
\usepackage{printlen}
\usepackage{algorithm}
\usepackage[noEnd,indLines,commentColor=gray]{algpseudocodex}

\usepackage{amsmath,amsthm,amsfonts}
\usepackage{amssymb,stmaryrd}
\usepackage{cleveref}
\usepackage{mathtools}
\usepackage{bbm}

\newtheorem{lemma}{Lemma}
\newtheorem{definition}{Definition}
\newtheorem{theorem}{Theorem}
\newtheorem{proposition}{Proposition}
\newtheorem{corollary}{Corollary}
\newtheorem{ass}{Assumption}
\newtheorem{rmk}{Remark}

\crefname{ass}{assumption}{assumptions}
\crefname{rmk}{remark}{remarks}

\DeclareMathOperator{\id}{I}
\newcommand{\expo}[1]{\exp\left(#1\right)}

\DeclareMathOperator{\divergence}{div}

\DeclareMathOperator{\KL}{KL}

\DeclareMathOperator{\trace}{tr}

\DeclareMathOperator{\cov}{cov}

\DeclareMathOperator{\law}{law}

\newcommand{\doublehookrightarrow}%
{\DOTSB\lhook\joinrel\relbar\!\!\!\!\lhook\joinrel\rightarrow}

\newcommand{\longrightharpoonup}%
{\relbar\joinrel\rightharpoonup}

\newcommand{\conditionalcomma}[1]{\ifx#1\empty\else,\fi}

\newcommand{\Pc}{\mathcal{P}}

\newcommand{\Fc}{\mathcal{F}}

\newcommand{\1}{\mathbbm{1}}
\newcommand{\E}{\mathbb{E}}
\renewcommand{\P}{\mathbb{P}}
\newcommand{\R}{\mathbb{R}}
\renewcommand{\E}{\mathbb{E}}
\newcommand{\N}{\mathbb{N}}

\newcommand{\grad}{\nabla}

\newcommand{\dd}{\mathrm{d}}

\renewcommand{\epsilon}{\varepsilon}

\newlength{\formulaindentwidth}

\newcommand{\Nc}{\mathcal{N}}

\newcommand{\Lc}{\mathcal{L}}

\newcommand{\step}{{h}}
\newcommand{\bstep}{{\mathbf{h}}}

\newcommand{\lip}{L}
\newcommand{\tamingconst}{M}
\newcommand{\lipdrift}{L_b}
\newcommand{\tamebound}{N}
\renewcommand{\t}{\top}
\newcommand{\lipB}{L}

\newcommand{\minB}{\beta_{\text{l}}}
\newcommand{\maxB}{\beta_{\text{u}}}
\newcommand{\ito}{It\^o}

\newcommand{\loc}{{\mathrm{loc}}}

\newcommand{\Cplsi}{C_\mathrm{PLSI}} 
\newcommand{\Clsi}{C_\mathrm{LSI}} 
\newcommand{\Cplsimin}{C^\mathrm{min}_\mathrm{PLSI}} 
\newcommand{\drift}{b}
\newcommand{\gauss}{\gamma}

\newcommand{\contlyapa}{\lambda}
\newcommand{\contlyapb}{\rho}

\newcommand{\nB}{{n_B}}
\newcommand{\holderB}{\delta}

\newcommand{\Mp}{M_p}  

\Crefname{chapter}{Chapter}{Chapters}
\crefname{chapter}{Cha.}{Chas.}
\Crefname{section}{Section}{Sections}
\crefname{section}{Sec.}{Secs.}
\Crefname{table}{Table}{Tables}
\crefname{table}{Tab.}{Tabs.}
\Crefname{definition}{Definition}{Definitions}
\crefname{definition}{Def.}{Defs.}
\Crefname{remark}{Remark}{Remarks}
\crefname{remark}{Rem.}{Rems.}
\Crefname{example}{Example}{Examples}
\crefname{example}{Ex.}{Exs.}
\Crefname{figure}{Figure}{Figures}
\crefname{figure}{Fig.}{Figs.}
\Crefname{equation}{Equation}{Equations}
\crefname{equation}{Eq.}{Eqs.}
\Crefname{algorithm}{Algorithm}{Algorithms}
\crefname{algorithm}{Alg.}{Algs.}

\usepackage{glossaries-extra}
\newabbreviation{ula}{ULA}{unadjusted Langevin algorithm}
\newabbreviation{ald}{ALD}{annealed Langevin dynamics}
\newabbreviation{myula}{MYULA}{Moreau-Yosida regularized unadjusted Langevin algorithm}
\newabbreviation{daz}{DAZ}{diffusion at absolute zero}
\newabbreviation{apgd}{APGD}{accelerated proximal gradient descent}
\newabbreviation{dsm}{DSM}{denoising score matching}
\newabbreviation{sde}{SDE}{stochastic differential equation}
\newabbreviation{map}{MAP}{maximum a-posteriori}
\newabbreviation{mcmc}{MCMC}{Markov chain Monte Carlo}
\newabbreviation{mc}{MC}{Markov chain}
\newabbreviation{mmse}{MMSE}{minimum mean-squared-error}
\newabbreviation{tdv}{TDV}{total deep variation}
\newabbreviation{mri}{MRI}{magnetic resonance imaging}
\newabbreviation{iid}{i.i.d.}{independent and identically distributed}
\newabbreviation{lsc}{l.sc.}{lower semicontinuous}
\newabbreviation{tv}{TV}{total variation}
\newabbreviation{bp}{BP}{belief propagation}
\newabbreviation{gmm}{GMM}{Gaussian mixture model}
\newabbreviation{kde}{KDE}{kernel density estimation}
\newabbreviation{kld}{KLD}{Kullback-Leibler divergence}
\newabbreviation{tvd}{TVD}{total variation distance}
\newabbreviation{sor}{SOR}{successive over-relaxation}
\newabbreviation{rv}{RV}{random variable}
\newabbreviation{ila}{ILA}{inertial Langevin algorithm}
\newabbreviation{em}{EM}{Euler-Maruyama}
\newabbreviation{nila}{NILA}{Nesterov's inertial Langevin algorithm}
\newabbreviation{skrock}{SK-ROCK}{SK-ROCK}
\newabbreviation{ou}{OU}{Ornstein-Uhlenbeck}
\newabbreviation{md}{MD}{molecular dynamics}
\newabbreviation{gd}{GD}{gradient descent}
\newabbreviation{hmc}{HMC}{Hamiltonian Monte Carlo}
\newabbreviation{ses}{SES}{stochastic exponential Euler scheme}
\newabbreviation{emd}{EMD}{earth mover's distance}
\newabbreviation{acf}{ACF}{autocorrelation function}
\newabbreviation{ot}{OT}{optimal transport}
\newabbreviation{rof}{ROF}{Rudin-Osher-Fatemi}
\newabbreviation{glm}{GLM}{Gaussian latent machine}
\newabbreviation{lm}{LM}{Leimkuhler-Matthews}
\newabbreviation{ode}{ODE}{ordinary differential equation}
\newabbreviation{ebm}{EBM}{energy-based model}
\newabbreviation{pes}{PES}{potential energy surface}
\newabbreviation{egnn}{EGNN}{equivariant graph neural network}
\newabbreviation{pdf}{PDF}{probability density function}
\newabbreviation{bfgs}{BFGS}{Broyden-Fletcher-Goldfarb-Shanno}
\newabbreviation{lbfgs}{L-BFGS}{Limited-memory BFGS}
\newabbreviation{mala}{MALA}{Metropolis-adjusted Langevin algorithm}
\newabbreviation{rml}{RML}{Riemannian Langevin dynamics}
\newabbreviation{mh}{MH}{Metropolis-Hastings}
\newabbreviation{pde}{PDE}{partial differential equation}
\newabbreviation{lsi}{LSI}{log-Sobolev inequality}
\newabbreviation{plsi}{PLSI}{preconditioned log-Sobolev inequality}
\newabbreviation{name}{TIPreL}{time-inhomogeneous preconditioned Langevin algorithm}
\newabbreviation{mld}{MLD}{Mirror Langevin dynamics}
\glsdisablehyper

\newcommand{\ie}{\textit{i.e.}}
\newcommand{\eg}{\textit{e.g.}}
\newcommand{\cf}{\textit{cf.}}

\newcommand{\Pot}{\Psi}

\begin{document}

\maketitle

\begin{abstract}
Langevin sampling from distributions of the form $p(x) \propto \exp(-\Pot(x))$ faces two major challenges: (global) \emph{mode coverage} and (local) \emph{mode exploration}.
The first challenge is particularly relevant for multi-modal distributions with disjoint modes, whereas the second arises when the potential $\Pot$ exhibits diverse and ill-conditioned local mode geometry.
To address these challenges, a common approach is to precondition Langevin dynamics with problem-specific information, such as the sample covariance or the local curvature of $\Pot$.
However, existing preconditioner choices inherently involve a trade-off between global mode coverage and local mode exploration, and no prior method resolves both simultaneously.
To overcome this limitation, we propose the \gls{name}, which introduces a time- and position-dependent preconditioner. 
This design effectively addresses both challenges mentioned above within a single framework.
We establish convergence of the resulting dynamics in the Wasserstein-2 distance both in continuous time and for a tamed Euler discretization.
In particular, our analysis extends the existing state of the art by proving convergence under time- and
space-dependent diffusion coefficients, and only locally Lipschitz drifts, which has not been covered by prior work.
Finally, we experimentally compare \gls{name} with competing preconditioning schemes on a two-dimensional, severely ill-posed example and on a Bayesian logistic regression task in higher dimensions, confirming the efficiency of the proposed method.
\end{abstract}

  
  

\section{Introduction}\label{sec:introduction}
We are interested in sampling from a probability distribution $\dd \pi(x) = p(x)\dd x$ on $\R^d$ with  $d\geq 1$, characterized via a density of Gibbs-Boltzmann form
\begin{equation}\label{eq:target}
    p(x) = \frac{\exp(- \Pot(x))}{\int_{\R^d} \exp(-\Pot(\xi)) \, \dd \xi}
\end{equation}
with potential function $\Pot \colon \R^d \to \R$.
Throughout, we assume that $\Pot$ is twice continuously differentiable with $\lip$-Lipschitz continuous gradient. More specifically, we are interested in sampling from $\pi$ via preconditioned Langevin dynamics with temporally and spatially varying preconditioning.

Sampling from Gibbs distributions of the above form is a prevalent task within machine learning and Bayesian inverse problems \cite{habring2025energy,narnhofer2022posterior,zach2021computed}.
Due to the complexity of $\Pot$ in most applications, direct sampling is not feasible, and researchers resort to \gls{mcmc} methods. Among those, the most popular ones are based on discretizations of the (overdamped) Langevin \gls{sde}
\begin{equation}\label{eq:langevin_sde}
    \dd X_t = - \grad \Pot (X_t) \, \dd t + \sqrt{2} \, \dd W_t.
\end{equation}
It is well-known that under mild conditions, such as Lipschitz-smoothness and certain growth assumptions on the potential, the law of the solution of~\eqref{eq:langevin_sde} converges to its unique stationary distribution $\pi$ as $t\rightarrow \infty$ in different metrics such as the \gls{kld} or the Wasserstein-2 distance~\cite{robertsExponentialConvergenceLangevin1996}.
Similarly, discretizing \eqref{eq:langevin_sde} via, \eg, an Euler-Maruyama discretization leads to an ergodic Markov chain whose stationary distribution approximates $\pi$ as the discretization step size vanishes~\cite{dalalyanTheoreticalGuaranteesApproximate2017a,durmus2017nonasymptotic,robertsExponentialConvergenceLangevin1996}.

However, Langevin sampling suffers from slow convergence, especially in high-dimensional and/or multi-modal settings.
While there exist several techniques for acceleration of Langevin sampling, in particular based on annealing/tempering~\cite{chehab2024provable,cordero2025non,habring2026forward,habring2026diffusion} or momentum/inertia~\cite{cheng2018underdamped,falk2025inertial}, we believe that the gap to acceleration techniques in numerical optimization is still significant. 

In this article, we consider accelerating Langevin sampling by the incorporation of a time- and position-dependent preconditioning, that is, we consider for a symmetric, positive definite matrix $B(t,x)\in\R^{d\times d}$ the Langevin \gls{sde}
\begin{equation}\label{eq:precond-langevin}
    \dd X_t = [- B(t,X_t) \nabla \Pot (X_t) + \divergence B(t, X_t)] \, \dd t + \sqrt{2}B^{1/2}(t,X_t) \, \dd W_t.
\end{equation}
where the divergence of the matrix-field $B$ is defined as the row-wise divergence
\begin{equation}\label{eq:correction}
    (\divergence B(t,x))_i = \sum_{j=1}^d \partial_{x_j} B_{ij}(t,x),\quad i=1,\dots, d.
\end{equation}
Note that, in addition to the preconditioning appearing in front of the force $\nabla\Pot$ and the Brownian motion, the term $\divergence B$ within~\eqref{eq:precond-langevin} is necessary to maintain $\pi$ as a stationary distribution. Whenever $B(t,x)$ is independent of $x$, the divergence vanishes, yielding simpler dynamics.
In the case that $B(t,x)$ is, in fact, independent of $t$,~\eqref{eq:precond-langevin} is the so-called \gls{rml}, \ie, Langevin dynamics on a Riemannian manifold with metric $B^{-1}(x)$~\cite{girolami_riemann_2011}.
Let us denote for simplicity the drift as $b(t,x) = - B(t,x) \nabla \Pot (x) + \divergence B(t, x)$. As the space dependence of the preconditioner leads to a loss of global Lipschitz continuity of $b(t,x)$, we employ a tamed Euler discretization. More precisely, as proposed in~\cite{brosse2019tamed} we consider as a discrete scheme approximating~\eqref{eq:precond-langevin} 
\begin{equation}\label{eq:tamed_euler}
    Y_{t_{k+1}} =Y_{t_{k}} + \step_k b_{\step_k}(t_k,X_k) + \sqrt{2\step_k}B^{1/2}(t_k,Y_{t_{k}}) Z_k, \quad Z_k\sim\Nc(0,\id)
\end{equation}
with step size $\step_k>0$, time points $t_k = \sum_{\ell =0}^{k-1}\step_\ell$ and the tamed drift
\begin{equation}
    b_\step(t,x) = \frac{b(t,x)}{1+\step |b(t,x)|}.
\end{equation}
This process corresponds to a sampling at the time-steps $(t_k)_k$ of \gls{sde} defined piece-wise for $t\in (t_k,t_{k+1}]$ as
\begin{equation}
    \dd Y_t = b_{\step_k}(t_k,Y_{t_k})\dd t + \sqrt{2}B^{1/2}(t_k,Y_{t_{k}})\dd W_t.
\end{equation}
While the considered algorithm will be proven to converge to the target under rather general assumptions on $B(t,x)$, inspired by quasi-Newton methods in optimization, we will specifically consider approximations of the inverse Hessian $B(t,x)\approx \nabla^2 \Pot(x)^{-1}$ (see~\cref{sec:numerical} for details) as well as time-dependent convex combinations of the form $B(t,x) = \lambda_t B_1(x) + (1-\lambda_t) B_0$ where $B_0\in\R^{d\times d}$ is a global (\ie, position-independent) preconditioner, $B_1(x)\in \R^{d\times d}$ is the above mentioned approximation of the Hessian, and $\lambda_t\in [0,1]$ satisfies $\lambda_t\rightarrow 1$ as $t\rightarrow \infty$.

\paragraph{Contributions}
In summary, we provide the following contributions.
\begin{itemize}
    \item We propose \gls{name}, Langevin sampling with time- and position-dependent preconditioning. Specifically, we propose to use a time-dependent convex combination of the (estimated) covariance and a quasi-Newton matrix to combine global and local curvature information in an ideal manner.
    \item We provide a rigorous and complete convergence analysis of the proposed method. This convergence analysis provides a novelty beyond the considered problem of preconditioned Langevin sampling since we prove convergence of Langevin sampling via a tamed Euler-Maruyama discretization for locally Lipschitz drift and time-dependent drift and diffusion. Under such general conditions these convergence results have not been available yet to our knowledge.
    \item We provide numerical results in two dimensions as well as for a high-dimensional Bayesian logistic regression example. The numerical experiments confirm the theoretical results and the efficacy of the proposed preconditioning.
\end{itemize}

\paragraph{Outline}
The remainder of the article contains a discussion of related works in~\cref{sec:related-work}, a short overview over the theoretical analysis of the method in~\cref{sec:theory} with the detailed proofs in the appendix, \cref{sec:appendix_proofs}, and numerical results for low-dimensional, but severely ill-posed potentials as well as for higher-dimensional Bayesian logistic regression in~\cref{sec:numerical}.

\section{Related work}\label{sec:related-work}

In optimization, preconditioners are used to accelerate the gradient-based minimization of ill-conditioned objective functions (see, \eg, \cite[Chapter 5]{nocedal_numerical_2006}).
The observed gain in convergence speed has led to the adoption of various preconditioning schemes in the field of \gls{mcmc} methods~\cite{girolami_riemann_2011,hanson_posterior_1998,leimkuhler_ensemble_2018,martin_stochastic_2012,qi_hessian-based_2002,xifara_langevin_2014,zhang_quasi-newton_2011}.
The simplest way to improve the behavior of the potential function in the Langevin \gls{mcmc} case is to employ a \emph{global}, \ie, time- and space-invariant, preconditioner $B$~\cite{girolami_riemann_2011}.
There exist various proposals for choices of $B$ most prominently, (the empirical estimate of) the target covariance $B=\cov(X)$~\cite{leimkuhler_ensemble_2018} or the inverse of the expected Fisher information matrix $B = \E_{X \sim\pi}[\nabla^2 \Pot(X)]^{-1}$~\cite{titsias2023optimal,wang_solving_2025} where it should be noted that the latter is the provably optimal global preconditioner with respect to the expected squared jump distance~\cite{titsias2023optimal}.

\paragraph{Position-dependent preconditioning} Position dependent preconditioning is significantly more difficult due to the appearance of the divergence correction in~\eqref{eq:precond-langevin}. Nonetheless, multiple efforts in this direction have been made.
As mentioned above, if the preconditioner depends on the position $x$ but not on the time $t$, the dynamics are, in fact, an instance of \gls{rml}, which have been analyzed extensively in the literature~\cite{bharath2025sampling,gatmiry2022convergence,girolami_riemann_2011,zhan2026riemannian}. However, it is worth mentioning that several of the proposed approaches analyze the geometric Euler-Maruyama discretization which assumes access to the manifold exponential~\cite{gatmiry2022convergence,zhan2026riemannian}.
In several methods it has been proposed to avoid the (possibly expensive) computation of the correction term $\divergence B$ to maintain invariance of $\pi$ and instead metropolize the resulting biased scheme~\cite{martin_stochastic_2012,srinivasan2024high,xifara_langevin_2014}.

\paragraph{Time-inhomogeneous preconditioning}
A crucial contribution of the proposed work is the incorporation of time-dependence of the preconditioner. Sampling using time-inhomogeneous Langevin dynamics has recently received increased attention via, \eg, tempering~\cite{chehab2024provable}, annealed Langevin dynamics~\cite{baldassari2026dimension,cordero2025non,habring2026forward,song2019generative}, or Moreau envelope-based approximations of the target $\pi$~\cite{habring2026diffusion}. All of these approaches, however, differ to the considered setting in two aspects: first, there is no time- or space-dependence of the diffusion term in the mentioned works, and second, the considered Langevin dynamics in these articles do not admit an invariant (but only a limiting) distribution as they aim at a \emph{moving target} whereas in our case the invariant distribution remains the same independently of the value of $t$ in $B(t,x)$.

\section{Theoretical results}\label{sec:theory}

\subsection{Notation and preliminaries}
The set of all (Borel) probability measures on $\R^d$ is denoted as $\Pc(\R^d)$ and we denote the set of measures with finite $p$-th moment as $\Pc_p(\R^d)$. We denote the distribution of the solution of the continuous time dynamics~\eqref{eq:precond-langevin} as $\mu_t = \law(X_t)$ and the distribution of the approximation as $\hat \mu_{t}=\law (Y_t)$. We will denote the step sizes of the discretization as $\step_k$, one step size sequence as $\bstep = (\step_0,\step_1,\dots)$ and correspondingly $t_k = \sum_{\ell =0}^{k-1}\step_\ell$.
Moreover, we denote the \emph{solution operator} of~\eqref{eq:langevin_sde} as $(P_{s,t})_{s\leq t}$. That is, with $(X_t)_t$ denoting a solution of \eqref{eq:langevin_sde}, we have for any sufficiently smooth $f:\R^d\rightarrow \R$
\begin{equation}
    P_{s,t} f(x) = \E[f(X_t)|X_s=x].
\end{equation}
Analogously, we define the Markov kernels $R_\step(t)$ corresponding to the discrete chain~\eqref{eq:tamed_euler} via
\begin{equation}
    R_\step(t)f(x) = (2\pi)^{-d/2}\int f\big(x + \step b_\step(t,x) + \sqrt{2\step}B^{1/2}(t,x)\big) \expo{-\frac{|z|^2}{2}}\dd z
\end{equation}
We denote the concatenation of multiple discrete steps as $Q_{t_\ell,t_k}^\bstep = R_{\step_\ell}(t_\ell)R_{\step_{\ell+1}}(t_{\ell+1})\cdots R_{\step_k}(t_k)$. The action of the adjoints of $P_{s,t}$, $R_\step (h)$, and $Q_{t_\ell,t_k}^\bstep$ on measures $\mu\in\Pc(\R^d)$ is denoted as $\mu P_{s,t}$, $\mu R_\step (t)$, and $\mu Q_{t_\ell,t_k}^\bstep$. 
As is customary, for $\mu\in\Pc(\R^d)$ and $f:\R^d\rightarrow \R$ we denote $\mu(f)\coloneqq \int f\dd \mu$. For any $a,b\in\R$ we denote $\min\{a,b\} = a\wedge b$ and $\max\{a,b\} = a\vee b$.

\subsection{Main results}\label{ssec:newton-langevin}
Before stating and discussing the main results we list the necessary assumptions.
\begin{ass}\label{ass}\
    \begin{enumerate}
        \item The potential $\Pot\in C^2(\R^d)$ with $\nabla \Pot$ $\lip$-Lipschitz. Without loss of generality $\nabla \Pot (0)=0$.
        \item $B(t,x)$ is bounded above and below, \ie, $\minB \id \preceq B(t,x)\preceq \maxB \id$ for some $\minB,\maxB>0$.
        \item $B(t,x)$ is twice continuously differentiable with respect to $x$ and both $B$ and $\divergence B$ are locally Hölder in time and locally Lipschitz in space in the following sense: There exist $\minB, \maxB, \lipB, \nB >0$, $\holderB\geq 1/2$ such that for all $x,y\in \R^d$, $s,t>0$ it holds
        \begin{equation}
            \begin{aligned}
                |B(t,x)-B(t,y)|\vee |\divergence B(t,x)-\divergence B(t,y)|&\leq \lipB (1+|x|^\nB + |y|^\nB) |x-y|\\
                |B(s,x)-B(t,x)|\vee |\divergence B(s,x)-\divergence B(t,x)|&\leq \lipB (1+|x|^\nB) |s-t|^\holderB
            \end{aligned}
        \end{equation}

        \item Denoting $b(t,x) = -B(t,x)\nabla \Pot (x) + \divergence B(t,x)$, we have 
        \begin{equation}
            \begin{aligned}
                \lim_{r\rightarrow\infty}\sup_{|x|\geq r,\, t\geq 0}\left\langle \frac{x}{|x|}, \frac{b(t,x)}{|b(t,x)|}\right\rangle&<0,\quad \text{and }
                \lim_{r\rightarrow\infty}\inf_{|x|\geq r,\, t\geq 0}|b(t,x)|=\infty.
            \end{aligned}
        \end{equation}
    \end{enumerate}
\end{ass}
\begin{rmk}\
    \begin{enumerate}

        \item The proofs work identically, when $\nabla \Pot$ is also only locally Lipschitz similar to $B$, \ie, for a $n\in \N$
        \begin{equation}
                |\nabla\Pot(x)-\nabla\Pot(y)|\vee |\divergence B(t,x)-\divergence B(t,y)|\leq \lipB (1+|x|^n + |y|^n) |x-y|.
        \end{equation}
        \item Of course, the local Lipschitz constants and the growth constants $\lip, \nB$ could be different for the different assumptions above, but are assumed the same for notational simplicity.
        \item The last assumption is in particular satisfied, if $\divergence B$ is bounded and 
        \begin{equation}
            \begin{aligned}
                \lim_{r\rightarrow\infty}\sup_{|x|\geq r,\, t\geq 0}\left\langle \frac{x}{|x|}, \frac{B(t,x) \nabla \Pot (x)}{|B(t,x) \nabla \Pot (x)|}\right\rangle&<0
                ,\quad \lim_{r\rightarrow\infty}\inf_{|x|\geq r,\, t\geq 0}|B(t,x)\nabla \Pot(x)|=\infty.
            \end{aligned}
        \end{equation}
        This assumption, in turn, can be ensured by requiring the identical assumption on the potential without the preconditioning $B$ together with some conformity, that is, the angle-preserving property of the preconditioner $B$.
    \end{enumerate}
\end{rmk}
Inspired by~\cite{chewi_exponential_2020}, we additionally define a \gls{plsi} as follows.
\begin{definition}[Preconditioned log-Sobolev inequality]
    We say that $\nu$ satisfies \gls{plsi} with preconditioner $B:\R^d\rightarrow\R^{d\times d}$, if there exists $\Cplsi$ such that
    \begin{equation}\label{eq:LSI}
        \int f^2\log\frac{f^2}{\int f^2\dd\nu}\dd\nu\leq \Cplsi \int |\nabla f|^2_B\dd\nu, \quad \text{for all locally Lipschitz $f\in L^2(\nu)$.}
        \tag{PLSI}
    \end{equation}
\end{definition}
As a second assumption, we require the target $\pi$ to satisfy such a \gls{plsi}.
\begin{ass}\label{ass:plis}
    For all $t$, $\pi$ satisfies \gls{plsi} with preconditioner $B(t,x)$ and constant $\Cplsi(t)$.
\end{ass}
\begin{rmk}\
    \begin{enumerate}
        \item Note that in our setting due to boundedness and uniform positive definiteness, the \gls{plsi} is equivalent to the regular \gls{lsi}, \ie, with preconditioner $B=\id$. However, we expect $\Cplsi$ to be smaller for well-chosen $B(t,x)$ in comparison to the case $B=\id$.
        \item An \gls{lsi} is, in particular satisfies, when the potential is dissipative, that is, $\langle\nabla\Pot(x),x\rangle\geq a|x|^2 + b$ for some $a>0$, $b\in \R$ and all $x\in\R^d$.
    \end{enumerate}
\end{rmk}

As a first result we can derive exponential convergence as in the setting without preconditioning. However, in this case the convergence rate will be improved for well-chosen $B(t,x)$ due to the smaller value of $\Cplsi$.
\begin{theorem}[Exponential convergence of the continuous-time dynamics]\label{thm:cont_time exp convergence}
    Let~\cref{ass,ass:plis} hold. Then, the preconditioned Langevin dynamics~\eqref{eq:precond-langevin} satisfy
    \begin{equation}
        \KL(\mu_t|\pi) \leq \exp\bigg(-\int_0^t\frac{4}{\Cplsi(s)}\dd s\bigg)\KL(\mu_0|\pi).
    \end{equation}
\end{theorem}
\begin{proof}
    See~\cref{sec:appendix_proof_cont_ergodic}.
\end{proof}
The next crucial result is a bound on the discretization error. This result is significantly more involved in the considered case for two reasons: The lack of global Lipschitz continuity of the drift and the non-constant Brownian motion. The former is notoriously known to lead to instability of basic Langevin-based sampling and is countered by the taming (\cf~\cite{brosse2019tamed}). The latter, in turn, prohibits the application Girsanov's theorem or coupling techniques due to the discrepancy in the Brownian motion between the continuous-time dynamics and its discretization
(\cf~\cite{brosse2019tamed,durmus2017nonasymptotic}).
\begin{theorem}[Discretization error]\label{lemma_discretization_error}
    Assume $X_0=Y_0\sim\mu$. For every $\alpha$ there exist constants $C_\alpha, C_1, C_2(R), C_3(R)>0$ such that for any $\bstep$, $T>0$, and $t\leq T$ it holds true that
    \begin{equation}\label{eq:discretization_statement}
        \E[|X_{t}- Y_{t}|^2]\leq C_1\sqrt{2\expo{-\alpha (R -C_\alpha T)} \mu(V_\alpha)} + \expo{TC_2(R)t}T C_3(R)\sum_{\ell\geq 0}\1_{t_\ell\leq t}\step_\ell^2
    \end{equation}
    where $k = \min\{\ell\,|\, t_\ell\geq t\}$.
    In particular, for constant step sizes, \ie, $\step_k = \step$ for all $k$
    \begin{equation}
        \E[|X_{t}- Y_{t}|^2]\leq C_1\sqrt{2\expo{-\alpha (R -C_\alpha T)} \mu(V_\alpha)} + \expo{TC_2(R)t}T^2 C_3(R)\step.
    \end{equation}
\end{theorem}
\begin{proof}
    See~\cref{sec:appendix_proof_discretization}.
\end{proof}

\begin{rmk}
    We want to briefly comment on the the structure of the above result. Due to the non-global Lipschitz continuity of the drift, within the proof of~\cref{lemma_discretization_error} we consider the escape times
    \begin{equation}
        \tau_R = \inf\{t\geq 0\,|\, |X_t|\geq R\},\quad\text{and}\quad k_R = \inf\{k\in \N\,|\, |Y_{t_k}|\geq R\}.
    \end{equation}
    By definition, whenever $t\leq \tau_R\wedge t_{k_R}$, we can restrict the analysis to the set $\{x\in\R^d\,|\, |x|< R \}$ on which the drift is Lipschitz. On the other hand, in~\cref{sec:appendix discrete} we derive the exponential bounds
    \begin{equation}
        \P[t_{k_R}\leq T] \vee \P[\tau_R\geq T]
        \leq \expo{-\alpha((1+R^2)^{1/2} - C_\alpha T)} \mu (V_\alpha)
    \end{equation}
    where $V_\alpha(x) = \expo{(1+|x|^2)^{1/2}}$ for $\alpha>0$ and $C_\alpha>0$. Thus, in~\eqref{eq:discretization_statement} the first term precisely reflects to the bound on the probabilities of the escape times being small and the second term the discretization error before the escape times.
\end{rmk}

\begin{theorem}[Convergence of the discrete scheme]\label{thm:convergence}
    The discrete scheme satisfies the following:
    \begin{enumerate}
        \item If the step sizes satisfy $\step_\ell\leq 1$ and they are not summable but square-summable, that is, $\sum_{\ell\geq 0}\step_\ell = \infty$, $\sum_{\ell\geq 0}\step_\ell^2 < \infty$
        then it holds $\lim_{k\rightarrow\infty} W_2(\mu_0 Q^\bstep_{0,t_k},\pi)=0$.
        \item If the step size is constant, \ie, $\step_\ell=\step$ for all $\ell$, then for every $\epsilon$, there exists $\delta>0$ such that for $\step<\delta$, it holds $\limsup_{k\rightarrow\infty} W_2(\mu_0 Q^\bstep_{0,t_k},\pi)<\epsilon$.
    \end{enumerate}
\end{theorem}
\begin{proof}
    See~\cref{proof:thm_convergence}
\end{proof}

Finally, \Cref{thm:convergence} establishes the convergence of the proposed tamed Euler discretization in \cref{eq:tamed_euler} in the Wasserstein-2 distance, concluding the theoretical analysis.

\section{Numerical experiments}\label{sec:numerical}
In this section, we consider several numerical experiments\footnote{All the numerical experiments were conducted on an NVIDIA GeForce RTX 4090 GPU.} comparing various preconditioner choices and, in particular, demonstrate that time- and position-dependent preconditioning leads to advantageous convergence.
Concretely, we compare the following choices for $B(t,x)$:
\begin{itemize}
    \item \textbf{Constant scalar}: $B(t, x)= 1/\lip$, where $\lip$ is the Lipschitz constant of $\nabla\Pot$ or the maximum spectral norm $\|\nabla^2 \Pot \|_2$ observed in a representative domain.
    This effectively constitutes a worst-case estimate of the potential curvature and thus serves as a baseline method.
    Note that, choosing a discretization step size $h=1$, this method is equivalent to performing \gls{ula} with a step size of $1/L$.
    \item \textbf{Global covariance}: $B(t, x) \approx \cov_\pi[X] = \E_\pi[(X-\E_\pi[X])(X-\E_\pi[X])^\t]$, which we empirically estimate from samples from $\pi$.
    As the covariance of a random variable characterizes its spatial spread along different directions, we expect this preconditioning scheme to accelerate the mixing rate of the simulated chains, leading to faster global mode coverage.
    However, in most practical scenarios, it is hard to obtain a good estimate of the covariance.
    A strategy to overcome this limitation by local approximations is presented in~\cite{leimkuhler_ensemble_2018}.
    \item \textbf{Inverse expected Fisher information}: Analogously, we also estimate $B(t, x) \approx \E_{\pi}[\nabla^2 \Pot(X)]^{-1}$ from ground truth data.
    Opposed to the global covariance strategy, this preconditioning scheme describes the expected curvature of the potential we want to sample from.
    It is expected that the local behavior of the simulated chains will improve. An assumption also formalized in~\cite{titsias2023optimal}.
    However, as in the previous case, this strategy will break down when sampling from potentials with diverse local geometries.
    \item \textbf{Curvature-aware preconditioning}: Further, based on the potential Hessian $\nabla^2 \Pot$, we consider a position-dependent preconditioner of the form $B(t, x) = Q(x) \Lambda^{-1}(x) Q(x)^\top$, with $\Lambda(x) = \diag(|\lambda_1|_\varepsilon, \dots, |\lambda_d|_\varepsilon)$, and $\lambda_i$ the eigenvalues of $\nabla^2 \Pot(x)$ belonging to the eigenvectors in $Q(x)$.
    Moreover, $|\cdot|_\varepsilon = \max \{|\cdot|, \varepsilon\}$ with $\varepsilon > 0$ denotes an approximation of the absolute value clamped away from zero.
    Similar techniques have been applied for Hamiltonian Monte Carlo in~\cite{betancourt2013general}.
    We hypothesize that this preconditioning scheme excels when sampling from potentials with locally diverse geometry.
    However, approximating this preconditioner without access to $\nabla^2 \Pot$ remains an open problem~\cite{simsekli_stochastic_nodate}.
    Also, the position-dependence necessitates the incorporation of the correction term from \cref{eq:correction}, which in general is computationally demanding.
    \item \textbf{Global-local interpolation}: Finally, to combine the strengths of global and position-dependent preconditioning, we propose $B(t, x) = (1 - \lambda_t) \cov_\pi[X] + \lambda_t Q(x) \Lambda^{-1}(x) Q(x)^\top$ with $\lambda_t \in [0,1]$ monotonically increasing.
    This constitutes a time-dependent preconditioning that evolves from a global covariance preconditioner to a geometry-adaptive approximate Hessian preconditioner, combining global and local curvature information in an advantageous manner.
    The global-to-local nature of this preconditioner is motivated by the fact that inverse-Hessian preconditioning is likely more effective when samples are already close to a mode, which will be the case at later stages of the dynamics.
    Conversely, at the beginning of the sampling, we emphasize \emph{mode coverage} by utilizing the covariance.
    In practice, we set the schedule $\lambda_t = \min \{ \frac{2t}{Kh}, 1\}$, where $K$ is the total number of Langevin steps and $h$ denotes the (constant) discretization step size.
    
\end{itemize}

\subsection{Two-dimensional Rosenbrock potential}
\begin{figure}
    \centering
    \includegraphics[width=1.0\linewidth]{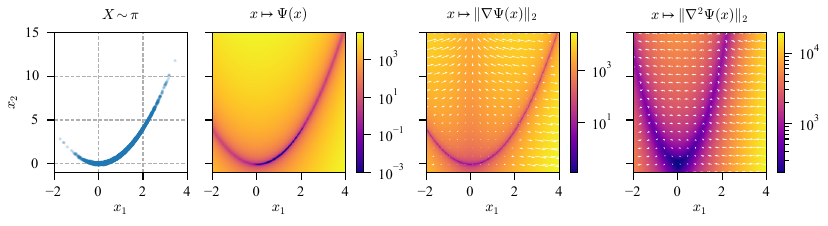}
    
    \caption{Two-dimensional Rosenbrock potential.
    The characteristic banana-shaped valley exhibits highly non-uniform curvature, posing a well-known challenge for gradient-based sampling methods.
    \emph{Left to right}: \emph{(i)} Ground truth samples obtained via ancestral sampling (\cf~\cref{appendix:details_rosenbrock}). \emph{(ii)} The potential energy landscape. \emph{(iii)} The gradient field of the potential. \emph{(iv)} The local curvature, with scaled Hessian eigenvectors overlaid at each point to illustrate the direction and magnitude of principal curvatures.
    This underlines the diverse local curvature encountered when sampling from the Rosenbrock potential.}
    \label{fig:rosenbrock}
\end{figure}
To showcase the advantage of spatiotemporal preconditioning, we start with sampling from the two-dimensional distribution induced by the Rosenbrock function, an ill-conditioned benchmark frequently encountered in numerical optimization~\cite{rosenbrock_automatic_1960}.
The potential takes on the form $\Pot(x_1, x_2) = (a - x_1)^2 + b (x_2 - x_1^2)^2$,
where we use the standard parameters $a=1$ and $b=100$ (see \cref{fig:rosenbrock} for details).
Note that this potential is only locally $L$-smooth on $\R^d$. Thus, for the constant, scalar preconditioning we estimate $L = \max_{x \in C} \|\nabla^2 \Pot(x)\|_2 \approx \num{11655}$ with $C = [-2, 4] \times [-1, 15]$ which captures 99\% of the mass.
Both the covariance and the expected Fisher information metric have been approximated via Monte Carlo estimation using ground truth samples which are available for the Rosenbrock potential.\footnote{see \cref{appendix:details_rosenbrock} for details}

To compare the proposed preconditioners, we simulate $N=\num{2e4}$ parallel chains with a fixed compute budget of $K=\num{1e4}$ Langevin steps.
The initial state of the trajectories is drawn from $Y_0 \sim \Nc(0, \id)$.
Further, we sweep log-linearly over discretization step sizes $h \in [\num{1e-4},1]$ to assess the trade-off between convergence speed and discretization bias for the different dynamics (see \cref{fig:rosenbrock-sweep} left).
Afterward, we select the step size $h = \num{6e-3}$ based on the given step budget, as in this regime most of the dynamics perform well.
We visualize the evolution of the Wasserstein-2 distance between the marginals for the compared dynamics in \cref{fig:rosenbrock-sweep} on the right.
The final samples produced by each of the considered preconditioned dynamics can be seen in \cref{fig:rosenbrock-final-samples}.

Additionally, we consider the expectations of different observables $\hat \mu_{t_k}(f)$ under the dynamics and compare their deviations from the ground-truth expectation over iterations.
Concretely, we consider $f(x) = x$ and $f(x) =f_{\gamma_1, \gamma_2}(x) = \cos(\gamma_1 x_1 + \gamma_2 x_2)$ for $(\gamma_1, \gamma_2) \in \{0,1,2\}^2$ (see \cref{fig:rosenbrock-stats}).
Finally, to obtain an impression of the temporal correlation structure between the samples produced by the compared dynamics at different times, we also report the \gls{acf} over different lags in \cref{fig:rosenbrock-acf}.

We observe faster convergence in the Wasserstein-2 metric and lower bias for \emph{Curvature} and \emph{Interpolated} throughout the experiment.
Furthermore, it becomes clear that time-dependent interpolation between global covariance and local Hessian-based preconditioning is beneficial even in the unimodal setting.
We hypothesize that this is due to the dynamics' better exploratory behavior at the beginning of the sampling process, followed by a switch to position-dependent step-size scaling.
\begin{figure}
    \centering
    \includegraphics[width=1.0\linewidth]{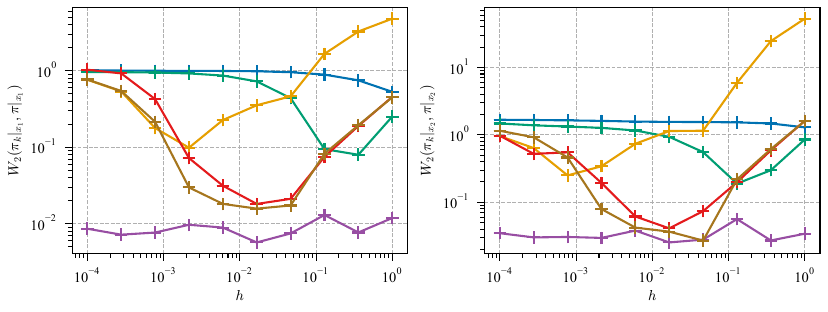}\\
    \includegraphics[width=1.0\linewidth]{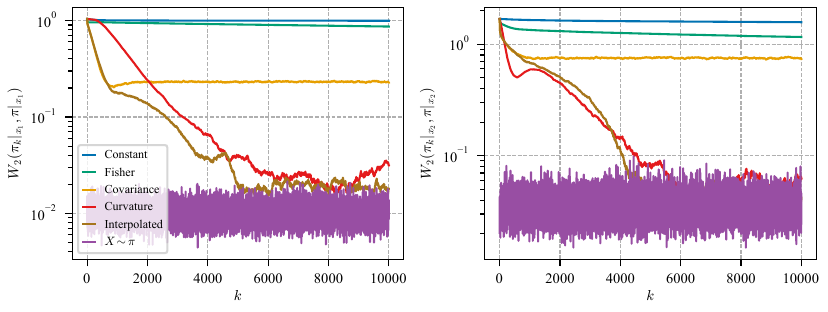}
    \caption{Performance of the investigated dynamics on the Rosenbrock potential. \emph{Top}: Sweep over the discretization step size $h$ for a fixed budget of $K=\num{1e4}$ Langevin steps. It is evident that the dynamics induced by the curvature-aware and interpolated preconditioners offer the best trade-off between speed and discretization bias. 
    \emph{Bottom}: Evolution of the marginal $W_2$-distances over the Langevin iterations for the optimal step size. \emph{Curvature} and \emph{Interpolated} significantly are best.}
    \label{fig:rosenbrock-sweep}
\end{figure}

\subsection{Bayesian logistic regression}
\label{sec:bayes_log_ref}
We now evaluate the proposed method on a higher-dimensional Bayesian logistic regression problem \cite{choi2008bayesianlogreg}.
Consider a set of pairs \(\{x_i, y_i\}_{i=1}^n\) where \(x_1, \dots, x_n \in \mathbb{R}^d\) are feature vectors, and \(y_1, \dots, y_n \in \{0,1\}\) are binary response variables.
In the Bayesian logistic regression setting, we treat each \(y_i\) as a Bernoulli random variable with parameter \(\varphi(\beta^\top x_i)\), where \(\varphi(u)= 1 / (1 + \exp(-u))\) for \(u \in \mathbb{R}\) and \(\beta \in \mathbb{R}^d\) are the regression parameters. 
The quantity of interest is the posterior distribution \(p({\beta \mid \{x_i, y_i\}_{i=1}^n)} \propto p(\{x_i, y_i\}_{i=1}^n \mid \beta)) \cdot p(\beta)\) over the regression parameters. We consider an anisotropic Gaussian prior \((p(\beta) \propto \exp(-\frac{1}{2}\|\beta\|^2_{\Sigma^{-1}})\) where \(\Sigma = \text{diag}(\sigma_1^2,\dots,\sigma_d^2)\) with linear interpolated variances between \(\sigma_1^2 = 0.1\) and \(\sigma_d^2 = 10.0\). This renders the posterior ill-conditioned. Specifically, the posterior potential is given as 
\begin{equation}
    \label{eq:ble_posterior_potential}
    \Psi(\beta \mid \{x_i, y_i\}_{i=1}^n) =     \sum_{i=1}^N \log(1 + \exp(\beta^\top x_i)) - y_i \beta^\top x_i + \frac{1}{2}\beta^\top\Sigma^{-1}\beta + c.
\end{equation}

For sampling from the posterior we consider the discrete dynamics in \eqref{eq:tamed_euler} with the preconditioners introduced in \cref{sec:numerical}. The gradient of the potential in \eqref{eq:ble_posterior_potential} is given as
\begin{equation}
\label{eq:ble_grad_potential}
\nabla \Psi(\beta \mid \{x_i, y_i\}_{i=1}^n) = X^\top(\varphi(X\beta) - y) + \Sigma^{-1} \beta,
\end{equation}
where the \emph{design matrix} \(X \in \mathbb{R}^{N \times d}\) contains all feature vectors and \(\varphi\) acts element-wise. The potential fulfills \Cref{ass}, and the Lipschitz constant of $\nabla\Pot$ is upper bounded by \(\lambda_{\max}(X^\top X) + \max_i 1 / \sigma_i^2\).

We evaluate our method on the heart disease dataset (\(d=13\)) from the University of California, Irvine repository \cite{dua2017logregdata}. We remove observations with missing values resulting in \(N = 297\) samples. To compare the preconditioners, we compute (i) the squared norm of the difference between the Monte Carlo estimate of the mean and the true posterior mean, and (ii) the Wasserstein-2 distance between the one-dimensional marginals of the sample distribution and the posterior distributions. To this end, we estimate the ground-truth posterior mean by running \(5\times 10^5\) iterations of the \gls{mala} \cite{robertsExponentialConvergenceLangevin1996} with \(1\times 10^4\) parallel chains. Following \cite{durmus2018analysisoflangevin}, we choose the step size such that the acceptance ratio is approximately \(0.5\). Furthermore, we use the final \(1\times 10^4\) samples to calculate the reference posterior marginals and to estimate the covariance- and the inverse expected Fisher-information-matrix. Throughout this section, we use \(1\times 10^4\) parallel chains in the inference.
\begin{figure}
    \centering
    \includegraphics[width=0.49\linewidth]{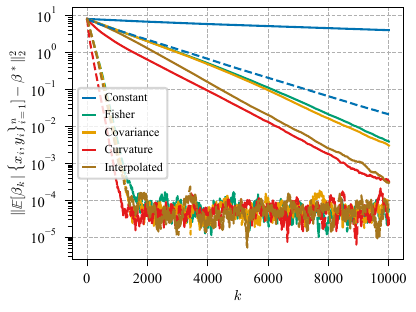}
    \includegraphics[width=0.49\linewidth]{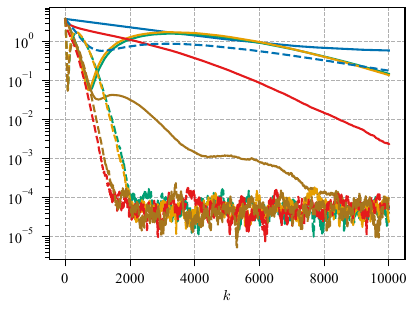}
    \caption{Performance of the investigated dynamics on the Bayesian logistic regression posterior under two initializations: \emph{left}\,\(\delta_1\); \emph{right}, \(\mathcal{N}(0,2 \id)\). Dashed lines correspond to $h=\num{5e-3}$ and solid lines to $h=\num{5e-4}$. The curvature-aware and interpolated preconditioners exhibit favorable convergence to the posterior mean across both step sizes and initializations
    compared to the global ones. This is especially pronounced for the Gaussian initialization (\emph{right}).} 
    \label{fig:blr_mean_convergence}
\end{figure}
\begin{figure}
    \centering
    \includegraphics[width=0.49\linewidth]{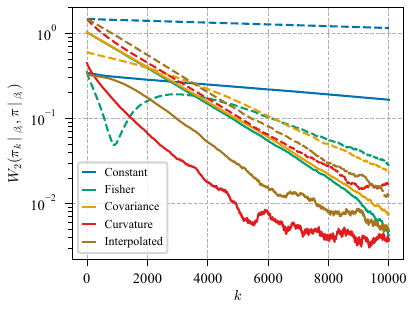}
    \includegraphics[width=0.49\linewidth]{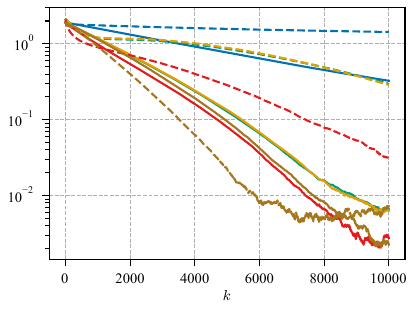}
    \caption{Performance of the investigated dynamics on the Bayesian logistic regression posterior under two initializations: \emph{left}\,\(\delta_1\); \emph{right}, \(\mathcal{N}(0,2 \id)\) for $h=\num{5e-4}$. Dashed / solid lines correspond to marginals with the highest / lowest Wasserstein-2 distance, after $\num{1e4}$ iterations, for the respective method. The curvature-aware and interpolated preconditioners show better convergence for both cases in this setting.} 
    \label{fig:blr_w2_convergence}
\end{figure}

We report results for two different initializations in \Cref{fig:blr_mean_convergence} and \Cref{fig:blr_w2_convergence}.
In both cases, the curvature-aware and interpolated preconditioners exhibit faster convergence to the posterior mean (see \Cref{fig:blr_mean_convergence}).
For a step size (\(h=\num{5e-3}\)), all methods except the constant preconditioner converge in roughly $\num{2000}$ iterations, with global preconditioners being slightly slower than the time- and/or position-dependent ones.
The advantage of our method becomes apparent when reducing the step size (\ie, \(h=\num{5e-4}\)): preconditioners based on the covariance matrix and inverse Fisher information exhibit problematic behavior, with the distance increasing after an initial dip under Gaussian initialization.
In contrast, the interpolated preconditioner combines fast initial convergence from the global component with local curvature-awareness to avoid an increase in distance.
In \Cref{fig:blr_w2_convergence}, we show the best and worst marginal Wasserstein-2 distances for the respective methods and \(h=\num{5e-4}\).
For both initializations, the convergence of the best and worst marginal of the position- and/or time-dependent preconditioners is comparable to, or even better than, that of the global ones.
We refer to \Cref{appendix:add_results_blr} for additional results. 

\section{Conclusion}
In this work, we have introduced a time- and position-dependent preconditioning framework 
for Langevin-based sampling and established convergence guarantees under general conditions.
In particular, the proven convergence results are of interest beyond the proposed applications as they apply more broadly to Langevin sampling with locally Lipschitz continuous drifts and time- and position-dependent drifts and diffusions.
In numerical experiments, we have shown improved sampling performance with position- or position- and time-dependent preconditioning for Langevin sampling compared to constant preconditioning.

\subsection{Limitations and future research}\label{sec:limitations}
The main limitation of the proposed approach is the computation of the divergence $\divergence B$.
In future research, we would like to investigate the approximate computation of this term in greater detail.
Moreover, inspired by quasi-Newton methods in optimization, like L-BFGS~\cite{nocedal_numerical_2006}, we would like to explore different efficient approximate Hessian preconditioners such as diagonal or learned preconditioners.

\bibliographystyle{plainnat}
{\small
    \bibliography{references}
}

\appendix

\section{Postponed proofs}\label{sec:appendix_proofs}

In the following we provide a detailed convergence analysis of the proposed preconditioned Langevin algorithm. The analysis will be structured as follows: In~\cref{sec:appendix continuous} we prove well-posedness of the preconditioned \gls{sde} and exponential convergence to the target in $\KL$. Afterward in~\cref{sec:appendix discrete} we bound the discretization error of the tamed Euler-Maruyama discretization. 
For these proofs we will make use of several helper results, which we append afterward in~\cref{sec:appendix growth cont,sec:appendix growth discrete,sec:appendix auxiliary}.
In particular, in~\cref{sec:appendix growth cont,sec:appendix growth discrete} we provide some crucial results such as Lyapunov drift conditions and moment bounds for the continuous and discrete dynamics, respectively and in~\cref{sec:appendix auxiliary} we present several more elementary results about continuity and growth proerties of the involved functions.

\subsection{Analysis of the continuous-time dynamics}\label{sec:appendix continuous}

\begin{proposition}\label{prop:preconditioned-well-posed}
    The \gls{sde}~\eqref{eq:precond-langevin} admits a unique strong global solution. Moreover, the law of the solution denoted as $\mu_t = \law(X_t)$ satisfies the weak \gls{pde}
    \begin{equation}\label{eq:Fokker}
        \partial_t \mu_t = \divergence(B \nabla \Pot \mu_t + B\nabla \mu_t)
    \end{equation}
    and admits a density with respect to the Lebesgue measure, that is, $\dd \mu_t(x) = q_t(x)\dd x$ for a.e. $t$.
\end{proposition}
\begin{proof}
    Existence of a unique, strong, global solution follows from~\cite[Chapter 2, Theorems 3.5, 3.6]{mao2007stochastic}. 

    In order to establish the existence of a density let us derive the corresponding Fokker-Planck equation. 
    For any test function $\phi\in C^\infty_c((0,\infty)\times\R^d)$ applying \ito's lemma to $\phi(X_t)$ yields
    \begin{equation}
        \begin{aligned}
            \phi(t,X_t) =& \int_0^t \partial_t \phi(s,X_s) + \nabla\phi(s,X_s)\cdot [- B(s,X_s) \nabla \Pot (X_s) + \divergence B(s,X_s)] \\
            &+ \trace\bigg( B^{\t/2}(s,X_s)\nabla^2\phi(s,X_s)B^{1/2}(s,X_s)\bigg)\dd s
            + \int_0^t\nabla\phi(s,X_s)^\t B^{1/2}(s,X_s) \dd W_s.
        \end{aligned}
    \end{equation}
    Taking the expectation, letting $t\rightarrow\infty$, and using that for any $C,D\in\R^{d\times d}$ it holds $\trace(CDC)=\trace (CCD)$ we find
    \begin{equation}\label{eq:FP1}
        \begin{aligned}
            0
            =& \int_0^\infty \int_{\R^d} \partial_t \phi(s,x) + \nabla\phi(s,x)\cdot [- B(s,x) \nabla \Pot (x) + \divergence B(s,x)]\\
            &+ \trace\bigg( B(s,x)\nabla^2\phi(s,x)\bigg) \dd \mu_s(x)\dd s\\
            =& \int_0^\infty \int_{\R^d} \partial_t \phi(s,x) + \nabla\phi(s,x)\cdot [- B(s,x) \nabla \Pot (x) + \divergence B(s,x)]\\
            &+ \trace\bigg( B(s,x)\nabla^2\phi(s,x)\bigg) \dd \mu_s(x)\dd s
        \end{aligned}
    \end{equation}
    By~\cite[Corollary 6.3.2, Section 9.4]{bogachev2022fokker} and the assumptions on $B$, $\mu = q_t(x)\dd x\dd t$ for some strictly positive $q\in L^r_\loc((0,\infty)\times\R^d)$ for every $r\in[1,(d+2)')$. Moreover, note that
    \begin{equation}
        \divergence(B\nabla \phi) 
        = \sum_{i,j} \partial_{x_i} (B_{i,j}\partial_{x_j}\phi)
        = \sum_{i,j} \partial_{x_i}B_{i,j} \partial_{x_j}\phi + B_{i,j}\partial_{x_i}\partial_{x_j}\phi
        = \divergence(B)\cdot\nabla \phi + \trace(B\nabla^2 \phi)
    \end{equation}
    so that~\eqref{eq:FP1} can equivalently be written as
    \begin{equation}
        \begin{aligned}
            0
            =& \int_0^\infty \int_{\R^d} \partial_t \phi(s,x) - \nabla\phi(s,x)\cdot B(s,x) \nabla \Pot (x) + \divergence \big(B(s,x)\nabla\phi(s,x)\big) \dd \mu_s(x)\dd s,
        \end{aligned}
    \end{equation}
    that is, $\mu$ satisfies in the weak sense~\eqref{eq:Fokker}.
\end{proof}

\begin{lemma}
    Let $\nu$ satisfy \gls{plsi} and assume $\nu$ admits a strictly positive and continuously differentiable density with respect to the Lebesgue measure. Then it holds
    \begin{equation}
        \KL(\mu|\nu)\leq \frac{\Cplsi}{4} \int |\nabla\log\frac{\dd\mu}{\dd\nu}|_B^2\dd\mu.
    \end{equation}
    where the right-hand side is set to $\infty$ whenever $\frac{\dd \mu}{\dd\nu}$ is not strictly positive a.e. or not contained in $W^{1,2}(\R^d)$.
\end{lemma}
\begin{proof}
    See~\cite[Lemma B.7]{habring2026forward}.
\end{proof}
\subsubsection{Proof of \cref{thm:cont_time exp convergence}}\label{sec:appendix_proof_cont_ergodic}

\begin{proof}
    The proof is virtually identical to the one without preconditioning but provided for the sake of completeness. We obtain for the time derivative of the $\KL$ along the dynamics
    \begin{equation}
        \begin{aligned}
            \frac{\dd }{\dd t} \KL(\mu_t|\pi)
            = \frac{\dd}{\dd t}\int q_t \log\big(\frac{q_t}{p}\big)\dd x
            =& \int \log \frac{q_t}{p} \partial_t q_t\dd x\\
            =& \int \log\frac{q_t}{p}\divergence(B \nabla \Pot q_t + B\nabla q_t)\dd x\\
            =& -\int \nabla \log\frac{q_t}{p}\cdot B ( -\nabla \log p q_t + \nabla q_t)\dd x\\
            =& -\int \nabla \log\frac{q_t}{p}\cdot B ( -\nabla \log p q_t + \nabla q_t)\dd x\\
            =& -\int |\nabla \log\frac{q_t}{p}|^2_B q_t\dd x\\
            \leq& -\frac{4}{\Cplsi(t)}\KL(\mu_t|\pi)
        \end{aligned}
    \end{equation}
    which implies
    \begin{equation}
        \begin{aligned}
            \frac{\dd}{\dd t}&\bigg(\exp\big(\int_0^t\frac{4}{\Cplsi(s)}\dd s\big)\KL(\mu_t|\pi) \bigg)\\
            &=\exp\big(\int_0^t\frac{4}{\Cplsi(s)}\dd s\big)\big(\frac{4}{\Cplsi(t)}\KL(\mu_t|\pi) + \frac{\dd}{\dd t}\KL(\mu_t|\pi)\big)\leq 0.
        \end{aligned}
    \end{equation}
    Integrating over $t$ then yields the desired result.
\end{proof}
\begin{rmk}
    For a rigorous treatment of the the proof of~\cref{thm:cont_time exp convergence} using the weak formulation of the Fokker-Planck equation we refer to~\cite{habring2026forward}.
\end{rmk}

\subsection{Discretization analysis}\label{sec:appendix discrete}
The analysis of the discretization error of the proposed scheme is substantially more involved for two reasons. First, the fact that the diffusivity is time- and space-dependent leads to a discrepancy also in the discretization of the Brownian motion which prohibits the application of Girsanov's theorem as in~\cite{brosse2019tamed}. On the other hand-side, the loss of global Lipschitz continuity of the drift also renders classical discretization error bounds difficult. As a remedy, we will in the following provide discretization errors restricted to compact subsets of $\R^d$ combined with quantitative bounds on the escape times from these sets.

We begin with the definitions of the following stopping times:
\begin{definition}[Escape times]
    Let $X_t$ be a solution of~\eqref{eq:precond-langevin} and $(Y_{t_k})_k$ a discretization thereof according to~\cref{eq:tamed_euler}.
    For $R>0$ define we define the following stopping times
    \begin{equation}
        \tau_R = \inf\{t\geq 0\,|\, |X_t|\geq R\},\quad\text{and}\quad k_R = \inf\{k\in \N\,|\, |Y_{t_k}|\geq R\}.
    \end{equation}
\end{definition}

Note that, within the event 
\begin{equation}\label{eq:time_event}
    \{t\leq \tau_R\wedge k_R\}
\end{equation}
both the continuous- as well as the discrete-time dynamics are bounded so that the local Lipschitz continuity assumed in~\cref{ass} effectively reduces to a global Lipschitz continuity with Lipschitz modulus depending on $R$. 

In the following we will derive quantitative bounds on the probability of the events \emph{contrary to \eqref{eq:time_event}}
\begin{equation}
    \{\tau_R\geq T\},\quad \text{and } \{k_R\geq T\}
\end{equation}
for $T>0$.
\begin{lemma}[Growth bound on $\tau_R$]\label{lemma:escape_time_cont}
    For every $\alpha>0$ there exists $C_\alpha>0$ such that $(X_t)_t$ initialized at $X_0\sim \mu$ satisfies
    \begin{equation}
        \P[\tau_R\leq T]\leq \expo{-\alpha ((1+R^2)^{1/2} - C_\alpha T)} \mu(V_\alpha).
    \end{equation}
\end{lemma}
\begin{proof}
    Consider $\psi(x) = (1+|x|^2)^{1/2}$. Since $|\nabla^2\psi|\leq 2$, via \ito's lemma we find
    \begin{equation}\label{eq_time_cont}
        \dd \psi(X_t) \leq \bigg\{ \left\langle\frac{X_t}{(1+|X_t|^2)^{1/2}},b(t,X_t)\right\rangle + 2d\maxB\bigg\}\dd t + \frac{\sqrt{2}X_t^\t}{(1+|X_t|^2)^{1/2}} B^{1/2}(t,X_t) \dd W_t.
    \end{equation}
    By~\cref{ass} it there exist $C>0$ such that
    \begin{equation}
        \langle x,b(t,x)\rangle\leq C.
    \end{equation}
    Inserting this into~\eqref{eq_time_cont} yields
    \begin{equation}
        \begin{aligned}
            \dd \psi(X_t) 
            \leq& \bigg\{\frac{C}{(1+|X_t|^2)^{1/2}} + 2d\maxB\bigg\}\dd t + \frac{\sqrt{2}X_t^\t}{(1+|X_t|^2)^{1/2}} B^{1/2}(t,X_t) \dd W_t.
        \end{aligned}
    \end{equation}
    Denoting $V_t = \int_0^t\frac{\sqrt{2}X_s^\t}{(1+|X_s|^2)^{1/2}} B^{1/2}(s,X_s) \dd W_s$ and noting that 
    \begin{equation}
        \langle V\rangle_t \leq 2\maxB t
    \end{equation}
    we obtain
    \begin{equation}
        \begin{aligned}
            \psi(X_t)
            \leq& \psi(X_0) + (C+\alpha\maxB)t + V_t - \frac{\alpha}{2}\langle V\rangle_t.
        \end{aligned}
    \end{equation}
    Since 
    \begin{equation}
        \left|\frac{\sqrt{2}X_s^\t}{(1+|X_s|^2)^{1/2}} B^{1/2}(s,X_s)\right|^2\leq 2\maxB
    \end{equation}
    is bounded by a constant, Novikov's condition is satisfied and we can apply~\cref{lemma:novikov} to the process $(V_t)_t$ yielding that
    \begin{equation}
        \begin{aligned}
            \P[\sup_{t\leq T}(1 + |X_t|^2)^{1/2}\geq R]
            \leq& \P[(1 + |x_0|^2)^{1/2} + (C+\alpha\maxB)T + \sup_{t\leq T}V_t - \frac{\alpha}{2}\langle V\rangle_t\geq R]\\
            \leq& \P[\sup_{t\leq T}V_t - \frac{\alpha}{2}\langle V\rangle_t\geq R - (1 + |x_0|^2)^{1/2} - (C+\alpha \maxB)T]\\
            \leq& \expo{-\alpha (R - (1 + |x_0|^2)^{1/2} - (C+\alpha \maxB)T)}\wedge 1\\
            \leq& \expo{-\alpha (R - (1 + |x_0|^2)^{1/2} - (C+\alpha \maxB)T)}.
        \end{aligned}
    \end{equation}
    Defining $C_\alpha = C+\alpha \maxB$ and integrating $x_0$ over $\mu$ yields the desired result as $\{\sup_{t\leq T}|X_t|\geq R\} = \{\sup_{t\leq T}(1+|X_t|^2)^{1/2}\geq (1+R^2)^{1/2}\} = \{\tau_R\leq T\}$.
\end{proof}
Using similar techniques, we obtain a bound on the discrete chain.
\begin{lemma}[Growth bound on $k_R$]\label{lemma:escape_time_discrete}
    Let $(Y_{t_k})_k$ be initialized at $Y_{0}\sim\mu$. For every $\alpha>0$ there exists a constant $C_\alpha >0$ such that for any $\step,T>0$ and $n=\lfloor T/\step \rfloor$ it holds
    \begin{equation}
        \P[k_R\leq n] 
        \leq \expo{-\alpha((1+R^2)^{1/2} - C_\alpha T)} \mu (V_\alpha)
    \end{equation}
\end{lemma}
\begin{proof}
    Recall that
    \begin{equation}
        \begin{aligned}
            Y_{t_{k+1}} 
            &= Y_{t_k} + \step_k \drift_{\step_k}(t_k,Y_{t_k}) + \sqrt{2{\step_k}}B^{1/2}(t_k,Y_{t_k})Z_k.
        \end{aligned}
    \end{equation}
    We find
    \begin{equation}\label{eq:discrete_stopping_time1}
        \begin{aligned}
            |Y_{t_{k+1}}|^2
            \leq& |Y_{t_k} + {\step_k} \drift_{\step_k}(t_k,Y_{t_k})|^2 \\
            &+ 2{\step_k} |B^{1/2}(t_k,Y_{t_k})Z_k|^2 + 2\langle Y_{t_k} + {\step_k} \drift_{\step_k}(t_k,Y_{t_k}), \sqrt{2{\step_k}}B^{1/2}(t_k,Y_{t_k})Z_k\rangle.\\
        \end{aligned}
    \end{equation}
    By~\cref{lemma:growth_tamed} we may pick $r$ such that for $|x|\geq r$ it holds for any $\step\leq \step_{\max}$
    \begin{equation}
        2 \bigg\langle \frac{x}{|x|},b_{\step}(t,x)\bigg\rangle + \frac{{\step}}{|x|}|b_{\step} (t,x)|^2\leq 0.
    \end{equation}
    which yields for $|x|\geq r$
    \begin{equation}
        |x + \step \drift_\step(t,x)|^2 = |x|^2 + |x|\step\big( 2 \langle \frac{x}{|x|}, \drift_\step(t_k,x)\rangle + \frac{\step}{|x|} |\drift_\step(t_k,x)|^2\big) \leq |x|^2.
    \end{equation}
    On the other hand-side, by~\cref{lemma_taming_bound}, there exists a constant $C_r$ independent of $\step$ such that for $|x|\leq r$,
    \[
        |x + \step \drift_\step(t,x)|^2 \leq |x|^2 + \step(2|x||\drift_\step(t,x)| + \step |\drift_\step(t,x)|^2)\leq |x|^2 + \step C_r
    \]
    so that we obtain altogether the bound
    \begin{equation}
        |x + \step \drift_\step(t,x)|^2\leq |x|^2 + \step C_r\1_{B_r(0)}(x)\leq |x|^2 + \step C_r.
    \end{equation}
    Inserting into~\eqref{eq:discrete_stopping_time1} yields
    \begin{equation}
        \begin{aligned}
            |Y_{t_{k+1}}|^2
            \leq& |Y_{t_k}|^2 +C_r{\step_k} \\
            &+ 2{\step_k} |B^{1/2}(t_k,Y_{t_k})Z_k|^2 + 2\langle Y_{t_k}+ {\step_k} \drift_{\step_k}(t_k,Y_{t_k}), \sqrt{2{\step_k}}B^{1/2}(t_k,Y_{t_k})Z_k\rangle.
        \end{aligned}
    \end{equation}
    Noting that by concavity of $s\mapsto (a+s)^{1/2}$, for every $a>0$ and $s>-a$ it holds $(a+s)^{1/2}\leq \sqrt{a} + \frac{s}{2\sqrt{a}}$, we find
    \begin{equation}\label{eq:discrete_stopping_time2}
        \begin{aligned}
            (1+|Y_{t_{k+1}}|^2)^{1/2}
            \leq& (1+|Y_{t_k}|^2)^{1/2} \\
            &+ \frac{C_r{\step_k}+ 2{\step_k} |B^{1/2}(t_k,Y_{t_k})Z_k|^2 + 2\langle Y_{t_k}+ {\step_k} \drift_{\step_k}(t_k,Y_{t_k}), \sqrt{2{\step_k}}B^{1/2}(t_k,Y_{t_k})Z_k\rangle}{2(1+|Y_{t_k}|^2)^{1/2}}\\
            \leq& (1+|Y_{t_k}|^2)^{1/2} + C_r/2{\step_k}+ {\step_k} |B^{1/2}(t_k,Y_{t_k})Z_k|^2 \\
            &+ \frac{\langle X_k+ {\step_k} \drift_{\step_k}(t_k,Y_{t_k}), \sqrt{2{\step_k}}B^{1/2}(t_k,Y_{t_k})Z_k\rangle}{(1+|Y_{t_k}|^2)^{1/2}}
        \end{aligned}
    \end{equation}
    Define now 
    \begin{equation}
        \xi_k = {\step_k} |B^{1/2}(t_k,Y_{t_k})Z_k|^2 - {\step_k}\trace(B(t_k,Y_{t_k})) + \frac{\langle Y_{t_k}+ {\step_k} \drift_{\step_k}(t_k,Y_{t_k}), \sqrt{2{\step_k}}B^{1/2}(t_k,Y_{t_k})Z_k\rangle}{(1+|Y_{t_k}|^2)^{1/2}}
    \end{equation}
    so that $\xi_k$ conditioned on $\Fc_{k-1}$ is a 1D Gaussian with mean zero. In particular, note that $M_n\coloneqq \sum_{k=0}^{n-1} \xi_k$ is a martingale and we can now write~\eqref{eq:discrete_stopping_time2} as
    \begin{equation}
        \begin{aligned}
            (1+|Y_{t_{k+1}}|^2)^{1/2}
            \leq& (1+|Y_{t_k}|^2)^{1/2} + C_r/2{\step_k} + {\step_k}\trace(B(t_k,Y_{t_k})) + \xi_k\\
            \leq& |Y_{t_k}|^2 + {\step_k}(C_r + d\maxB) + \xi_k.
        \end{aligned}
    \end{equation}
    Summing over $k$ then yields for any $\alpha>0$
    \begin{equation}
        \begin{aligned}
            (1+|Y_{t_k}|^2)^{1/2} - \frac{\alpha}{2} \langle M\rangle_k
            \leq& (1+|Y_0|^2)^{1/2} + \sum_{\ell = 0}^{k-1}\step_k(C_r + d\maxB) + \sum_{\ell=0}^{k-1}\xi_k - \frac{\alpha}{2} \langle M\rangle_k\\
            =& (1+|Y_0|^2)^{1/2} + t_{k}(C_r + d\maxB) + M_k - \frac{\alpha}{2} \langle M\rangle_k.
        \end{aligned}
    \end{equation}
    where 
    \begin{equation}
        \langle M\rangle_k = \sum_{\ell=0}^{k-1}\E[\xi_k^2|\Fc_{k-1}].
    \end{equation}
    Moreover, since we have for some constant $C>0$
    \begin{equation}
        \langle M\rangle_n\leq \sum_{k\leq n-1} C\step_k \leq Ct_n,
    \end{equation}
    we find for $\bstep,T>0$ arbitrary and $n$ such that $t_n\leq T$ with the notation $C_\alpha = C\alpha$
    \begin{equation}
        \begin{aligned}
            \sup_{k\leq n} (1+|Y_{t_k}|^2)^{1/2}
            \leq& \sup_{k\leq n} \bigg\{ (1+|Y_{t_k}|^2)^{1/2} - \frac{\alpha}{2}\langle M\rangle_k \bigg\} + C_\alpha T\\
            \leq& \sup_{k\leq n}\bigg\{(1+|Y_0|^2)^{1/2} + t_k(C_r + d\maxB) + M_k - \frac{\alpha}{2} \langle M\rangle_k\bigg\} + C_\alpha T\\
            \leq& \sup_{k\leq n} \big\{M_k - \frac{\alpha}{2} \langle M\rangle_k\big\} + (1+|Y_0|^2)^{1/2} + C_\alpha'T.
        \end{aligned}
    \end{equation}
    Assume again $Y_0 = x_0$ is deterministic. As shown in~\cite[Lemma A.1]{lamba2007adaptive} it holds true that
    \begin{equation}
        \P[\sup_{k\leq n} M_k - \frac{\alpha}{2} \langle M\rangle_k \geq \beta]\leq \expo{-\alpha\beta}.
    \end{equation}
    Thus, we can estimate
    \begin{equation}
        \begin{aligned}
            \P[\sup_{k\leq n}(1+|Y_{t_k}|^2)^{1/2}\geq R] 
            \leq& \P[\sup_{k\leq n} \big\{M_k - \frac{\alpha}{2} \langle M\rangle_k\big\} + (1+|x_0|^2)^{1/2} + C_\alpha' T\geq R]\\
            =& \P[\sup_{k\leq n} \big\{M_k - \frac{\alpha}{2} \langle M\rangle_k\big\} \geq R - (1+|x_0|^2)^{1/2} - C_\alpha' T] \\
            \leq& \expo{-\alpha(R - (1+|x_0|^2)^{1/2} - C'T)}
        \end{aligned}
    \end{equation}
    As previously, integrating with respect to $\mu(\dd x_0)$ yields the desired result.
\end{proof}

\subsubsection{Proof of \cref{lemma_discretization_error}}\label{sec:appendix_proof_discretization}

\begin{proof}
    We assume in the following for notational simplicity that $t=t_k$ for some $k\in\N$. The proof works analogously in the converse case.
    In the following we denote $\sigma_R = T\wedge t_{k_R}\wedge \tau_R$. We have
    \begin{equation}\label{eq_discretization1}
        \begin{aligned}
            \E[|X_t- Y_t|^2]
            &= \E[|X_t- Y_t|^2\1_{t\leq \sigma_R}] + \E[|X_t- Y_t|^2\1_{t>\sigma_R}]\\
            &\leq \E[|X_{t\wedge \sigma_R}- Y_{t\wedge \sigma_R}|^2] + \E[|X_t- Y_t|^2\1_{t>\sigma_R}].
        \end{aligned}
    \end{equation}
    We begin estimating the easier second term. Cauchy-Schwartz yields
    \begin{equation}\label{eq_discretization00}
        \E[|X_t- Y_t|^2\1_{t>\sigma_R}]\leq \E[|X_t- Y_t|^4]^{1/2}\E[\1_{t>\sigma_R}]^{1/2}.
    \end{equation}
    The first term is uniformly bounded by a constant, say $C_1$, for all $\bstep, T, t$ by the moment bounds~\cref{lemma:moment_bounds,cor:moment_bounds_cont}. For the second term we find
    \begin{equation}\label{eq_discretization0}
        \begin{aligned}
            \E[\1_{t\geq\sigma_R}]^{1/2} 
            &\leq \sqrt{\P[t\geq\tau_R] + \P[t\geq t_{k_R}]}
            \leq \sqrt{2\expo{-\alpha (R -C_\alpha T)} \mu(V_\alpha)}
        \end{aligned}
    \end{equation}
    where $C_\alpha$ is the greater (\ie, \emph{worse}) of the two constants appearing in \cref{lemma:escape_time_cont,lemma:escape_time_discrete}.
    Let us now tackle the first term in~\eqref{eq_discretization1}. Denoting $t_\ell \wedge t = t_\ell'$ and using Young's inequality and Cauchy-Schwartz multiple times we find
    \begin{equation}\label{eq_discretization2}
        \begin{aligned}
            \E[|X_{t\wedge \sigma_R}- Y_{t\wedge \sigma_R}|^2]
            =&\E\bigg[\bigg|\sum_{\ell \geq 0} \int_{t_\ell'\wedge \sigma_R}^{t_{\ell+1}'\wedge \sigma_R} \drift_\step(t_\ell, Y_{t_\ell}) - \drift(s,X_s)\dd s \\
            &\quad+ \sum_{\ell \geq 0} \sqrt{2}\int_{t_\ell'\wedge \sigma_R}^{t_{\ell+1}'\wedge \sigma_R} B^{1/2}(t_\ell, Y_{t_\ell}) - B^{1/2}(s,X_s)\dd W_s \bigg|^2 \bigg]\\
            \leq&\E\bigg[ 2t\sum_{\ell \geq 0} \int_{t_\ell'\wedge \sigma_R}^{t_{\ell+1}'\wedge \sigma_R} |\drift_\step(t_\ell, Y_{t_\ell}) - \drift(s,X_s)|^2\dd s \\
            &\quad+ 4\E \bigg[ \bigg|\sum_{\ell \geq 0} \int_{t_\ell'\wedge \sigma_R}^{t_{\ell+1}'\wedge \sigma_R} B^{1/2}(t_\ell, Y_{t_\ell}) - B^{1/2}(s,X_s)\dd W_s \bigg|^2 \bigg]\\
        \end{aligned}
    \end{equation}
    We first tackle the non-stochastic integral. Using~\cref{lemma_approx_taming} and the local Lipschitz assumption in~\cref{ass} we find
    \begin{equation}
        \begin{aligned}
            |\drift_\step(t_\ell, Y_{t_\ell}) &- \drift(s,X_s)|^2
            \leq \big(|\drift_{\step_\ell}(t_\ell, Y_{t_\ell}) - \drift(t_\ell, Y_{t_\ell})|
            + |\drift(t_\ell, Y_{t_\ell}) - \drift(s,X_s)|\big)^2\\
            \leq & \big({\step_\ell} \tamingconst (1+|Y_{t_\ell}|^{\nB+1}) \\
            &+ \lipdrift (1+|Y_{t_\ell}|^{\nB+1} + |X_s|^{\nB+1})(|Y_{t_\ell}-X_s| + |t_\ell-s|^\holderB)\big)^2\\
            \leq & \big({\step_\ell} \tamingconst (1+|Y_{t_\ell}|^{\nB+1}) \\
            &+ \lipdrift (1+|Y_{t_\ell}|^{\nB+1} + |X_s|^{\nB+1})(|Y_{t_\ell}-Y_s| + |Y_s - X_s| + |\step_{\ell}|^\holderB)\big)^2\\
            \leq & 2\step_\ell^2 \tamingconst^2 (1+|Y_{t_\ell}|^{\nB+1})^2 \\
            &+ 6\lipdrift^2 (1+|Y_{t_\ell}|^{\nB+1} + |X_s|^{\nB+1})^2(|Y_{t_\ell}-Y_s|^2 + |Y_s - X_s|^2 + |\step_{\ell}|^{2\holderB})\\
        \end{aligned}
    \end{equation}
    where we used that $(\sum_{i=1}^N a_i)^2\leq N\sum_{i=1}^N a_i^2$ for the last inequality.
    Note that 
    \begin{equation}
        \begin{aligned}
            Y_{s} - Y_{t_\ell}
            = (s-t_\ell)\drift_{\step_\ell}(t_\ell,Y_{t_\ell}) + \sqrt{2}B^{1/2}(t_\ell,Y_{t_\ell}) (W_s-W_{t_\ell})
        \end{aligned}
    \end{equation}
    implying with~\cref{lemma_taming_bound} and independence of Brownian increments
    \begin{equation}\label{eq:discretization_proof1}
        \begin{aligned}
            \E[|Y_{s} - Y_{t_\ell}|^2]
            \leq (s-t_\ell)^2 (1+\step_\ell)\tamingconst (1+\E[|Y_{t_\ell}|^{\nB+1}]) + 2 d \maxB (s-t_\ell)
        \end{aligned}
    \end{equation}
    Thus, we find
    \begin{equation}\label{eq:discretization_proof2}
        \begin{aligned}
            \E\bigg[\int_{t_\ell'\wedge \sigma_R}^{t_{\ell+1}'\wedge \sigma_R} &|\drift_{\step_\ell}(t_\ell, Y_{t_\ell}) - \drift(s,X_s)|^2\dd s\bigg]\\
            \leq& \E\bigg[\int_{t_\ell'}^{t_{\ell+1}'} 2\step_\ell^2 \tamingconst^2 (1+|Y_{t_\ell}|^{\nB+1})^2 \\
            &+ 6\lipdrift^2 (1+2R^{\nB+1})^2(|Y_{t_\ell}-Y_s|^2 + |\step_{\ell}|^{2\holderB})\dd s\\
            & + \int_{t_\ell'}^{t_{\ell+1}'}6\lipdrift^2 (1+2R^{\nB+1})^2|Y_{s\wedge \sigma_R}- X_{s\wedge \sigma_R}|^2\dd s\bigg].
        \end{aligned}
    \end{equation}
    Using~\eqref{eq:discretization_proof1} and the fact that by~\cref{ass}, $2\holderB\geq 1$ it follows
    \begin{equation}\label{eq:discretization_proof3}
        \begin{aligned}
            \E\bigg[\int_{t_\ell\wedge \sigma_R}^{t_{\ell+1}\wedge \sigma_R} |\drift_{\step_\ell}(t_\ell, Y_{t_\ell}) - &\drift(s,X_s)|^2\dd s\bigg]\\
            \leq& C_2(R) \step_{\ell}^2 + C_3(R) \int_{t_\ell'}^{t_{\ell+1}'} \E[|Y_{s\wedge \sigma_R}- X_{s\wedge \sigma_R}|^2]\dd s.
        \end{aligned}
    \end{equation}
    For the stochastic integral in~\eqref{eq_discretization2} let us briefly introduce the notation $\bar B(s,x) = B(t_\ell,x_{t_\ell})$ for $s\in [t_\ell, t_{\ell+1})$. Using \ito's isometry we obtain
    \begin{equation}
        \begin{aligned}
            \E \bigg[ \bigg|\sum_{\ell \geq 0} &\int_{t_\ell'\wedge \sigma_R}^{t_{\ell+1}'\wedge \sigma_R} B^{1/2}(t_\ell, Y_{t_\ell}) - B^{1/2}(s,X_s)\dd W_s \bigg|^2 \bigg]\\
            =&\E \bigg[ \bigg| \int_{0}^{t\wedge \sigma_R} \bar B^{1/2}(s, Y) - B^{1/2}(s,X_s)\dd W_s \bigg|^2 \bigg]\\
            =&\E \bigg[ \bigg| \int_{0}^{t} \1_{s\leq \sigma_R}(\bar B^{1/2}(s, Y) - B^{1/2}(s,X_s))\dd W_s \bigg|^2 \bigg]\\
            =&\E \bigg[ \int_{0}^{t} \1_{s\leq \sigma_R}\big|\bar B^{1/2}(s, Y) - B^{1/2}(s,X_s)\big|^2\dd s \bigg]\\
            =&\E \bigg[\sum_{\ell \geq 0} \int_{t_\ell'}^{t_{\ell+1}'} \1_{s\leq \sigma_R}\big|B^{1/2}(t_\ell, Y_{t_\ell}) - B^{1/2}(s,X_s)\big|^2\dd s \bigg].
        \end{aligned}
    \end{equation}
    For the latter in turn we find again due to the local Lipschitz continuity of $B^{1/2}$
    \begin{equation}
        \begin{aligned}
            \E \bigg[\int_{t_\ell'}^{t_{\ell+1}'} &\1_{s\leq \sigma_R}\big|B^{1/2}(t_\ell, Y_{t_\ell}) - B^{1/2}(s,X_s)\big|^2\dd s \bigg]\\
            \leq& \int_{t_\ell'}^{t_{\ell+1}'} \frac{1}{4\minB}\E[\1_{s\leq \sigma_R}\lipB^2 (1+|Y_{t_\ell}|^{\nB+1} + |X_s|^{\nB+1})^2(|Y_{t_\ell}-Y_s| + |Y_s - X_s| + |\step_{\ell}|^\holderB)^2]\dd s\\
            \leq& C_4(R) \step_{\ell}^2 + C_5(R) \int_{t_\ell'}^{t_{\ell+1}'} \E[|Y_{s\wedge \sigma_R}- X_{s\wedge \sigma_R}|^2]\dd s
        \end{aligned}
    \end{equation}
    where the last inequality follows using the same techniques as in~\eqref{eq:discretization_proof3}.
    Inserting the estimates into~\eqref{eq_discretization2} yields
    \begin{equation}
        \begin{aligned}
            \E[|X_{t\wedge \sigma_R}- Y_{t\wedge \sigma_R}|^2]
            \leq T C_6(R)\sum_{\ell\geq 0}\1_{t_\ell\leq t}\step_\ell^2 + TC_7(R)\int_{0}^{t} \E[|Y_{s\wedge \sigma_R}- X_{s\wedge \sigma_R}|^2]\dd s
        \end{aligned}
    \end{equation}
    and Grönwall's lemma gives us
    \begin{equation}\label{eq_discretization4}
        \begin{aligned}
            \E[|X_{t\wedge \sigma_R}- Y_{t\wedge \sigma_R}|^2]
            \leq \expo{TC_7(R)t}T C_6(R)\sum_{\ell\geq 0}\1_{t_\ell\leq t}\step_\ell^2 .
        \end{aligned}
    \end{equation}
    Combining~\eqref{eq_discretization00},~\eqref{eq_discretization0}, and~\eqref{eq_discretization4} we obtain
    \begin{equation}
        \E[|X_{t}- Y_{t}|^2]\leq C_1\sqrt{2\expo{-\alpha (R -C_\alpha T)} \mu(V_\alpha)} + \expo{TC_7(R)t}T C_6(R)\sum_{\ell\geq 0}\1_{t_\ell\leq t}\step_\ell^2
    \end{equation}
    concluding the proof up to relabeling of the constants.
\end{proof}

\subsubsection{Proof of~\cref{thm:convergence}}\label{proof:thm_convergence}

\begin{proof}
    We separate the error into one component from the discretization and one from exponential convergence of the continuous dynamics
    \begin{equation}
        W_2(\mu_0 Q^\bstep_{0,t_k},\pi)\leq W_2(\mu_0Q^\bstep_{0,t_{k_1}}Q^\bstep_{t_{k_1},t_k},\mu_0Q^\bstep_{0,t_{k_1}}P_{t_{k_1},t_{k}}) + W_2(\mu_0Q^\bstep_{0,t_{k_1}}P_{t_{k_1},t_{k}},\pi).
    \end{equation}
    By~\cite[Theorem 1]{otto2000generalization} \gls{lsi} implies Talagrand's inequality with the same constant, that is, for any $\mu\in\Pc_2(\R^d)$ we have
    \begin{equation}
        W_2^2(\mu,\pi)
        \leq \Clsi \KL(\mu|\pi).
    \end{equation}
    Using~\cref{thm:cont_time exp convergence} we find
    \begin{equation}\label{eq:convergence_thm1}
        \begin{aligned}
            W_2(\mu_0Q^\bstep_{0,t_{k_1}}P_{t_{k_1},t_{k}},\pi)
        \leq& \Clsi \KL(\mu_0Q^\bstep_{0,t_{k_1}}P_{t_{k_1},t_{k}}|\pi)\\
        \leq& \Clsi \exp\bigg(-\int_{t_{k_1}}^{t_k}\frac{4}{\Cplsi(s)}\dd s\bigg)\KL(\mu_0Q^\bstep_{0,t_{k_1}}|\pi)\\
        \leq& C\Clsi \exp\bigg(-\int_{t_{k_1}}^{t_k}\frac{4}{\Cplsi(s)}\dd s\bigg)
        \end{aligned}
    \end{equation}
    where used in the last inequality that $\KL(\mu_0Q^\bstep_{0,t_{k_1}}|\pi)$ is bounded uniformly with respect to $k_1,\bstep$ due to~\cref{lemma:bounded_KL}. Moreover, by~\cref{lemma_discretization_error} we have with $T=t_k-t_{k_1}$
    \begin{equation}
        \begin{aligned}
            W_2(&\mu_0Q^\bstep_{0,t_{k_1}}Q^\bstep_{t_{k_1},t_k},\mu_0Q^\bstep_{0,t_{k_1}}P_{t_{k_1},t_{k}})\\
            \leq& C_1\sqrt{2\expo{-\alpha (R -C_\alpha T)} \mu_0Q^\bstep_{0,t_{k_1}}(V_\alpha)} + \expo{TC_7(R)T}T C_6(R)\sum_{\ell\geq 0}\1_{t_{k_1}\leq t_\ell\leq t_{k-1} }\step_\ell^2.
        \end{aligned}
    \end{equation}
    Now let $\epsilon>0$ be arbitrary. Choose $\alpha<a_0$ with $a_0$ from~\cref{lemma:Lyapunov_discrete} and $T>0$ large enough so that 
    \begin{equation}
        C\Clsi \exp\bigg(-\int_{t}^{t + T}\frac{4}{\Cplsi(s)}\dd s\bigg)< \frac{\epsilon}{3}
    \end{equation}
    for all $t>0$.
    Afterward choose $R>0$ large enough so that
    \begin{equation}
        C_1\sqrt{2\expo{-\alpha (R -C_\alpha (T+1))} \mu_0Q^\bstep_{0,t_{k_1}}(V_\alpha)}< \frac{\epsilon}{3}
    \end{equation}
    which is possible as $\mu_0Q^\bstep_{0,t_{k_1}}(V_\alpha)$ is bounded by~\cref{lemma:moment_bounds}. Finally, by square summability, for fixed $T,R>0$ we can find $k_0$ such that for $k_1\geq k_0$
    \begin{equation}\label{eq_convergence_thm2}
        \expo{(T+1)C_7(R)(T+1)}(T+1) C_6(R)\sum_{\ell\geq 0}\1_{t_{k_1}\leq t_\ell\leq t_{k-1} }\step_\ell^2<\frac{\epsilon}{3}.
    \end{equation}
    It follows $W_2(\mu_0 Q^\bstep_{0,t_k},\pi)<\epsilon$ as soon as
    \begin{equation}
        k\geq \min\{\ell\geq k_0\,|\,t_\ell - t_{k_0}\geq T\}.
    \end{equation} 
    Indeed, for such $k$ we can define 
    \begin{equation}
        k_1 = \max\{\ell\leq k\,|\, t_k-t_\ell\geq T\}
    \end{equation}
    so that $k_1\geq k_0$ and $T\leq t_k-t_{k_1}\leq T+1$ by the fact that $\step_\ell\leq 1$. By the above choices of $T$, $R$, and $k_0$ the error bound of $\epsilon$ follows.
    In the constant step-size setting, the sum in~\eqref{eq_convergence_thm2} reduced to 
    \begin{equation}
        \sum_{\ell\geq 0}\1_{t_{k_1}\leq t_\ell\leq t_{k-1} }\step_\ell^2 = (t_k-t_{k_1})\step.
    \end{equation}
    Thus, choosing the step size as
    \begin{equation}
        \step\leq \frac{\epsilon}{2\expo{(T+1)C_7(R)(T+1)}(T+1)^2 C_6(R)}
    \end{equation}
    yields $\epsilon$ error $W_2(\mu_0 Q^\bstep_{0,t_k},\pi)\leq \epsilon$ whenever $k\geq \lceil \frac{T}{\step}\rceil$.
\end{proof}

\subsection{Drift condition and moment bounds for the continuous dynamics}\label{sec:appendix growth cont}
Based on the Fokker-Planck equation we define the following time-dependent generator of the diffusion 
\begin{equation}
    \Lc_t(V) = -\langle\nabla V, B\nabla\Pot\rangle + \divergence(B \nabla V).
\end{equation}
We can derive the following drift condition for $\Lc_t$.
\begin{lemma}\label{lemma:lyapunv_cont}
    Let $V_a(x) = \expo{a(1+|x|^2)^{1/2}}$. For any $a>0$, there exist $\contlyapa,\contlyapb>0$ such that for all $t\geq 0$
    \begin{equation}
        \Lc_t(V_a)(x)\leq -\contlyapa V_a(x) + \contlyapb.
    \end{equation}
\end{lemma}

\begin{proof}
    We compute
    \begin{equation}\label{eq:cont_lyapunov1}
        \begin{aligned}
            \Lc_t(V_a)(x) 
            =& \langle\nabla V_a(x), -B(t,x)\nabla\Pot(x) \rangle + \divergence(B\nabla V_a(x))\\
            =& \langle \nabla V_a(x), -B(t,x)\nabla\Pot(x)\rangle + \langle \divergence B(t,x), \nabla V_a(x)\rangle + \trace(B(t,x)\nabla^2 V_a(x))\\
            =& \langle \nabla V_a(x), b(t,x)\rangle + \trace(B(t,x)\nabla^2 V_a(x))\\
            =& \frac{aV_a(x)}{(1+|x|^2)^{1/2}} \langle x, b(t,x)\rangle + \trace(B(t,x)\nabla^2 V_a(x))\\
        \end{aligned}
    \end{equation}
    It holds $\trace(B\nabla^2 V_a)\leq \beta |\nabla^2V_a(x)|_{S^1}$ with $|\cdot|_{S^1}$ the Schatten-$1$ norm. We can compute
    \begin{equation}
        \nabla^2 V_a(x) 
        = aV_a (x) \left( \frac{axx^\t}{1+|x|^2} + \frac{(1+|x|^2)^{1/2}\id - x \frac{x^\t}{(1+|x|^2)^{1/2}}}{1+|x|^2}\right).
    \end{equation}
    It holds true that $\nabla^2 V(x)\succeq 0$ and we find for any $z\in \R^d$, $|z|\leq 1$
    \begin{equation}
        \begin{aligned}
            z^\t\nabla^2 V_a(x) z
            =& aV_a (x) \left( \frac{a|z^\t x|^2}{1+|x|^2} + \frac{(1+|x|^2)^{1/2}|z|^2 - \frac{|z^\t x|^2}{(1+|x|^2)^{1/2}}}{1+|x|^2}\right)\\
            \leq& aV_a (x) \left( a|z|^2 + \frac{|z|^2}{(1+|x|^2)^{1/2}}\right)\\
            \leq& a(1+a)V_a (x).
        \end{aligned}
    \end{equation}
    so that $|\nabla^2V_a(x)|_{S^1}\leq d a(1+a)V_a(x)$. Inserting into~\cref{eq:cont_lyapunov1} yields
    \begin{equation}
            \Lc_t(V_a)(x) 
            \leq V_a(x)\left(\frac{a}{(1+|x|^2)^{1/2}} \langle x, b(t,x)\rangle + \maxB d a(1+a)\right).
    \end{equation}
    By~\cref{ass} we may choose $r,\kappa>0$ uniformly in $t$ such that for $|x|\geq r$, 
    \begin{equation}
        \langle x, b(t,x)\rangle<-\kappa |x| |b(t,x)|.
    \end{equation}
    It follows that there exists some $C>0$ such that
    \begin{equation}
        \langle x, b(t,x)\rangle \leq -\kappa |x| |b(t,x)| + C\1_{B_r(0)}\leq -\kappa |x| |b(t,x)| + C
    \end{equation}
    Inserting above yields
    \begin{equation}\label{eq_lyapunov_cont}
        \begin{aligned}
            \Lc_t(V_a)(x) 
            \leq& V_a(x)\left(-\frac{a(\kappa |x| |b(t,x)| - C)}{(1+|x|^2)^{1/2}} + \maxB d a(1+a)\right)
        \end{aligned}
    \end{equation}
    implying the drift condition since
    \begin{equation}
        \frac{|x| |b(t,x)|}{(1+|x|^2)^{1/2}}\rightarrow\infty
    \end{equation}
    as $|x|\rightarrow\infty$ uniformly in $t$ by~\cref{ass}. 
\end{proof}

\begin{corollary}\label{cor:moment_bounds_cont}
    Let $V_a$ as in~\cref{lemma:lyapunv_cont}. If $\mu_0(V_a)<\infty$ then it holds true that
    \begin{equation}
        \sup_{t\geq 0}\mu_t(V_a) < \infty.
    \end{equation}
    In particular, all moments of the continuous dynamics are uniformly bounded for $t>0$. That is, for every $p\in \N$, 
    \begin{equation}
        \Mp \coloneqq \sup_{t>0}\int |x|^p\dd \mu_t(x) <\infty.
    \end{equation}
\end{corollary}
\begin{proof}
    By~\cref{lemma:lyapunv_cont}, and Grönwall's lemma we have for the process $\mu_t$ initialized at the dirac distribution $\delta_{x_0}$
    \begin{equation}
        \mu_t(V_a)
        \leq (V_a(x_0)-\contlyapb/\contlyapa)\expo{-\contlyapa t} + \contlyapb/\contlyapa
        = V_a(x_0)\expo{-\contlyapa t} + \contlyapb/\contlyapa (1-\expo{-\contlyapa t}).
    \end{equation}
    Integrating over any initial measure $\mu_0$ yields
    \begin{equation}
        \mu_t(V_a)
        \leq \mu_0(V_a)\expo{-\contlyapa t} + \contlyapb/\contlyapa (1-\expo{-\contlyapa t}).
    \end{equation}
    In particular, $(\mu_t(V_a))_t$ is bounded for all $t$ and since for every $p\in \N$, there exists $C_p$ such that $|x|^p\leq C_pV_a(x)$ the result follows.
\end{proof}
\subsection{Drift condition and moment bounds for the discrete dynamics}\label{sec:appendix growth discrete}
\begin{lemma}\label{lemma:Lyapunov_discrete}
    Define the Lyapunov function $V_a(x) = \expo{a(1+|x|^2)^{1/2}}$. There exists $a_0>0$ such that for every $a<a_0$ there are $\lambda,b,R>0$ such that for all $t,\step$ the Markov kernel $R_h(t)$ satisfies the drift condition
    \begin{equation}
        R_\step(t)V_a(x)\leq \expo{-\lambda\step}V_a(x) + \step b\1_{B_R(0)}(x).
    \end{equation}
\end{lemma}
\begin{proof}
    The proof closely follows~\cite{brosse2019tamed}. For any $x$ the map
    \[
        z\mapsto \phi(z) = (1+|x+\step b_\step(t,x) + \sqrt{2\step B(t,x)} z|^2)^{1/2}
    \]
    is $\sqrt{2\maxB\step}$-Lipschitz. Thus, by the log-Sobolev inequality for the Gaussian distribution $\frac{\dd \gamma}{\dd z}(z) = \frac{1}{(2\pi)^{d/2}}\expo{-\frac{|z|^2}{2}}$ (\cf~\cite[Theorem 5.5]{boucheron2003concentration}) and Cauchy-Schwartz we have that
    \begin{equation}\label{eq:lyapunov_discrete}
        \begin{aligned}
            R_\step(t)V_a(x) 
            =& \int \expo{a\phi(z)} \dd\gauss(z)\\
            \leq& \expo{a^2\maxB\step}\expo{a R_\step(t)\phi(x)}\\
            \leq& \expo{a^2\maxB\step}\expo{a\big(\int (1+|x+\step b_\step(t,x) + \sqrt{2\step B(t,x)} z|^2) \dd\gauss(z)\big)^{1/2}}\\
            \leq& \expo{a^2\maxB\step}\expo{a\big(1+|x+\step b_\step(t,x)|^2 + {2d\maxB\step} \big)^{1/2}}.
        \end{aligned}
    \end{equation}
    To bound the exponent we estimate
    \begin{equation}
        \begin{aligned}
            |x+\step b_\step(t,x)|^2
            =& |x|^2 + 2\step \langle x,b_\step(t,x)\rangle + \step^2|b_\step (t,,x)|^2\\
            =& |x|^2 + |x|\step \bigg(2 \bigg\langle \frac{x}{|x|},b_\step(t,x)\bigg\rangle + \frac{\step}{|x|}|b_\step (t,,x)|^2\bigg).
        \end{aligned}
    \end{equation}
    By~\cref{lemma:growth_tamed}, we can pick $r,\kappa$ independent of $\step>0$ such that 
    \begin{equation}
        \bigg(2 \bigg\langle \frac{x}{|x|},b_\step(t,x)\bigg\rangle + \frac{\step}{|x|}|b_\step (t,,x)|^2\bigg)\leq -\kappa
    \end{equation}
    for $|x|\geq r$. Potentially increasing $r$ so that $r>4d\maxB/\kappa$ we obtain
    \begin{equation}
        \begin{aligned}
            |x+\step b_\step(t,x)|^2 + 2d\maxB\step \leq |x|^2 - |x|\step( \kappa - 2d\maxB/|x|)\leq |x|^2 - |x|\step \kappa/2.
        \end{aligned}
    \end{equation}
    Note that for any $a>0$, the function $s\mapsto (a+s)^{1/2}$ is concave on $[-a,\infty)$ so that a linear approximation in $s=0$ yields for all $s\in [-a,\infty)$
    \[
        (a+s)^{1/2}\leq a^{1/2} + \frac{s}{2a^{1/2}}.
    \]
    Applied to $\big(1+|x|^2 - |x|\step \kappa/2 \big)^{1/2}$ with $a=1+|x|^2$ we get
    \begin{equation}
        \begin{aligned}
            \big(1+|x+\step b_\step(t,x)|^2 + {2d\maxB\step} \big)^{1/2}
            \leq& \big(1+|x|^2 - |x|\step \kappa/2 \big)^{1/2}\\
            \leq& \big(1+|x|^2 \big)^{1/2} - \frac{|x|\step \kappa/2}{2 \big(1+|x|^2 \big)^{1/2}}\\
            \leq& \big(1+|x|^2 \big)^{1/2} - \frac{r\step \kappa}{4 \big(1+r^2 \big)^{1/2}}
        \end{aligned}
    \end{equation}
    where in the last inequality we made use of the fact that $s\mapsto s/(1+s^2)^{1/2}$ is non-decreasing.
    Setting $a < \frac{r \kappa}{4\maxB \big(1+r^2 \big)^{1/2}}$ we obtain from~\eqref{eq:lyapunov_discrete}
    \begin{equation}
        R_h(t)V_a(x)\leq \expo{- ha\left(\frac{r \kappa}{4 \left(1+r^2 \right)^{1/2}}-a\maxB\right)}V_a(x).
    \end{equation}
    where $\lambda \coloneqq a\frac{r \kappa}{4 \big(1+r^2 \big)^{1/2}}-a\maxB>0$.
    On the other hand-side, we find
    \begin{equation}
        \begin{aligned}
            \big(1+|x+&\step b_\step(t,x)|^2 + {2d\maxB\step} \big)^{1/2}\\
            \leq&  \big(1+|x|^2+2\step |x||b_\step(t,x)| + \step^2 |b_\step(t,x)|^2 + {2d\maxB\step} \big)^{1/2}\\
            \leq&  \big(1+|x|^2\big)^{1/2}+\step\frac{2 |x||b_\step(t,x)| + \step |b_\step(t,x)|^2 + 2d\maxB}{2(1+|x|^2)^{1/2}}.
        \end{aligned}
    \end{equation}
    By~\cref{lemma_taming_bound}, there exists $C>0$ such that for all $\step\leq \step_{\max}$ and $|x|\leq R$
    \begin{equation}
        \frac{2 |x||b_\step(t,x)| + \step |b_\step(t,x)|^2 + 2d\maxB}{2(1+|x|^2)^{1/2}}\leq C.
    \end{equation}
    Therefore, for $|x|\leq C$ we have 
    \begin{equation}
        \begin{aligned}
            R_\step(t)V_a(x) 
            \leq& \expo{a^2\maxB\step}\expo{a\big(1+|x+\step b_\step(t,x)|^2 + {2d\maxB\step} \big)^{1/2}}\\
            \leq& \expo{a^2\maxB\step}\expo{a\big(1+|x|^2)^{1/2}+C\step}\\
            \leq& \expo{a\big(1+|x|^2)^{1/2}+(C+a^2\maxB)\step}\\
            =& V_a(x) \expo{(C+a^2\maxB)\step}
        \end{aligned}
    \end{equation}
    so that the elementary inequality $1-\expo{s}\leq s$ yields
    \begin{equation}
        \begin{aligned}
            R_\step(t)V_a(x) - V_a(x)&\expo{-\lambda\step}
            \leq V_a(x)\left(\expo{(C+a^2\maxB)\step} -\expo{-\lambda\step}\right)\\
            \leq& V_a(x)\expo{(C+a^2\maxB)\step}\left(1 -\expo{-(\lambda+C+a^2\maxB)\step}\right)\\
            \leq& V_a(x)\expo{(C+a^2\maxB)\step}(\lambda+C+a^2\maxB)\step\\
        \end{aligned}
    \end{equation}
    so that we obtain the desired result with 
    \begin{equation}
        b = \expo{(1+r^2)^{1/2}}\expo{(C+a^2\maxB)\step_{\max}}(\lambda+C+a^2\maxB).
    \end{equation}
\end{proof}

\begin{lemma}\label{lemma:moment_bounds}
    Let $a<a_0$ and $V_a$ as in~\cref{lemma:Lyapunov_discrete}. If $\mu_0(V_a)<\infty$, then
    \begin{equation}
        \sup_{\bstep,k}\mu_0Q^\bstep_{0,t_k}(V_a)<\infty.
    \end{equation}
    In particular, all moments of the discrete scheme are uniformly bounded, that is, for every $p\in\N$ there exists $C(p)>0$ such that 
    \begin{equation}
        \sup_{\bstep,t} \int |x|^p\dd\hat\mu_t(x)<C(p).
    \end{equation}
\end{lemma}
\begin{proof}
    We can write
    \begin{equation}
        \mu_0Q^\bstep_{0,t_k}(V_a) = \mu_0(R_{\step_1}(0)\cdots R_{\step_k}(t_{k-1})V_a).
    \end{equation}
    In turn, by~\cref{lemma:Lyapunov_discrete} we have
    \begin{equation}
        \hat{\mu}_0(R_{\step_0}(0)\cdots R_{\step_{k-1}}(t_{k-1})V_a(x))
        \leq \expo{-a^2t_k}\hat{\mu}_0(V_a) + b \sum_{\ell=0}^{k} \step_\ell \expo{-(t_k-t_{\ell})}
    \end{equation}
    so that the result follows if $\sum_{\ell=0}^{k} \step_\ell \expo{-(t_k-t_\ell)}<\infty$. Let us denote for a specific sequence $\bstep = (\step_0,\step_1,\dots)$
    \[
        S_k^{\bstep} = \sum_{\ell=0}^{k} \step_\ell \expo{-(t_k-t_\ell)}.
    \]
    We have $S_{k+1}^{\bstep} = \expo{-\step_{k+1}}S_k^{\bstep} + \step_{k+1}$. Consider the mapping $\phi_\step: \R\rightarrow\R$, $\phi_\step(s) = \expo{-\step}s + \step$. One can easily check that for any $\step$, $\phi_h$ is a contraction on $[0,\infty)$ with unique fixed point $\frac{h}{1-\expo{-h}}$ which converges to $1$ as $h\rightarrow0$. Moreover, $\phi(s)\leq s$ whenever $s\geq \frac{h}{1-\expo{-h}}$ and $\phi(s)\leq \frac{h}{1-\expo{-h}}$ whenever $s\leq \frac{h}{1-\expo{-h}}$. Thus, 
    \begin{equation}
        \phi(s)\leq \max\left\{s,\frac{h}{1-\expo{-h}}\right\}
    \end{equation}
    implying that 
    \begin{equation}
        \begin{aligned}
            \sup_{\bstep,k}S^{\bstep}_k\leq \sup_{\bstep,k}\max\left\{S^{\bstep}_{k-1},\frac{h_{k}}{1-\expo{-h_{k}}}\right\}
            &\leq \sup_{\bstep,k}\max\left\{S^{\bstep}_{k-1},\sup_{h\leq \bstep_{\max}}\frac{h}{1-\expo{-h}}\right\}\\
            &\leq \max\left\{\step_0,\sup_{h\leq \bstep_{\max}}\frac{h}{1-\expo{-h}}\right\}\\
            &<\infty.
        \end{aligned}
    \end{equation}
\end{proof}

\begin{lemma}\label{lemma:bounded_KL}
    The $\KL$ divergence between the iterates $\mu_0Q^\bstep_{0,t_k}$ and the target $\pi$ remains bounded, that is,
    \begin{equation}
        \sup_{\bstep,k}\KL(\mu_0Q^\bstep_{0,t_k}|\pi)<\infty.
    \end{equation}
\end{lemma}
\begin{proof}
    In the following we denote as $\hat q(x,t|y,s)$ the density of the distribution of $Y_t| Y_s=y$ where $Y_t$ solves 
    \begin{equation}
        \dd Y_t = \drift(s,Y_{s})\dd t + \sqrt{2 B(s,Y_s)}\dd W_s.
    \end{equation}
    The density $\hat{q}$ satisfies the Fokker-Planck equation
    \begin{equation}
        \begin{aligned}
            \partial_t \hat q (x,t|x_{t_k},t_k) 
            =& \divergence(B(t_k,x_{t_k}) \nabla \Pot (x_{t_k}) \hat q (x,t|x_{t_k},t_k)) + \divergence(B(t_k,x_{t_k})\nabla \hat q (x,t|x_{t_k},t_k) )\\
            =& \divergence((B(t_k,x_{t_k}) \nabla \Pot (x_{t_k}) + (B(t_k,x_{t_k})-B(t,x))\nabla \log\hat q (x,t|x_{t_k},t_k)) \hat q (x,t|x_{t_k},t_k)) \\
            &+ \divergence(B(t,x)\nabla \hat q (x,t|x_{t_k},t_k) )
        \end{aligned}
    \end{equation}
    Let us denote for simplicity the drift term as
    \begin{equation}
        D(x,t|x_{t_k},t_k)= B(t_k,x_{t_k}) \nabla \Pot (x_{t_k}) + (B(t_k,x_{t_k})-B(t,x))\nabla \log\hat q (x,t|x_{t_k},t_k)
    \end{equation}
    Via integration over $x_{t_k}$ it follows
    \begin{equation}
        \begin{aligned}
            \partial_t \hat q_t(x) 
            =& \int \partial_t \hat q (x,t|y,t_k)  \hat q_{t_k} (y) \dd y  \\
            =& \int \divergence(D(x,t|y,t_k)\hat q (x,t|y,t_k)) \hat q_{t_k} (y)\dd y
            + \int \divergence(B(t,x)\nabla \hat q (x,t|y,t_k) ) \hat q_{t_k} (y) \dd y  \\
            =& \divergence\bigg(\int D(x,t|y,t_k) \hat q (y,t_k|x,t) \dd y \hat q_{t} (x) \bigg)
            + \divergence\bigg(B(t,x)\nabla \int \hat q (x,t|y,t_k) \hat q_{t_k} (y) \dd y\bigg)  \\
            =& \divergence\bigg(\int D(x,t|y,t_k) \hat q (y,t_k|x,t) \dd y \hat q_{t} (x) \bigg) + \divergence\bigg(B(t,x)\nabla \hat q_t(x)\bigg).
        \end{aligned}
    \end{equation}
    Note that
    \begin{equation}
        \begin{aligned}
            |(B(t_k,x_{t_k})-&B(t,x))\nabla \log\hat q (x,t|x_{t_k},t_k)|\\
            &\leq L(|t_k-t| + |x_{t_k}-x|)|(4\step)^{-1} B^{-1}(t_k,x_{t_k})(x-(x_{t_k}+\step b_\step(t_k,x_{t_k}))|
        \end{aligned}
    \end{equation}
    Similarly to the continuous time case we can compute the time derivative of the $\KL$ along the iterates as
    \begin{equation}\label{eq:discrete1}
        \begin{aligned}
            \frac{\dd}{\dd t} \KL(\hat \mu_t|\pi)
            =& \frac{\dd}{\dd t} \int \hat q_t \log\frac{\hat{q}_t}{p}\dd x\\
            =& - \int \nabla \log\frac{\hat{q}_t}{p} \cdot \int D(x,t|y,t_k) \hat q (y,t_k|x,t) \hat q_{t} (x) \dd y  \dd x\\
            &-\int \nabla \log\frac{\hat{q}_t}{p}\cdot B(t,x)\nabla \hat q_{t} (x) \dd x\\
            =& - \int \nabla \log\frac{\hat{q}_t}{p} \cdot \int B(t_k,x) \nabla \Pot (x) \hat q (y,t_k|x,t) \hat q_{t} (x) \dd y  \dd x\\
            &-\int \nabla \log\frac{\hat{q}_t}{p}\cdot B(t,x)\nabla \hat q_{t} (x) \dd x\\
            &- \int \nabla \log\frac{\hat{q}_t}{p} \cdot \int (B(t_k,y) \nabla \Pot (y)-B(t_k,x) \nabla \Pot (x)) \hat q (y,t_k|x,t) \hat q_{t} (x) \dd y  \dd x\\
            &- \int \nabla \log\frac{\hat{q}_t}{p} \cdot \int \bigg\{(B(t_k,y)-B(t,x))\nabla \log\hat q (x,t|y,t_k)\bigg\} \hat q (y,t_k|x,t) \hat q_{t} (x) \dd y \dd x\\
            &= A+B+C+D.
            \end{aligned}
    \end{equation}
    where we \emph{added a zero} to account for the discretization in the drift. For the first two terms we find
    \begin{equation}
        \begin{aligned}
            A+B
            =&- \int \nabla \log\frac{\hat{q}_t}{p} \cdot B(t_k,x) \nabla \Pot (x) \hat q_{t} (x) \dd x
            -\int \nabla \log\frac{\hat{q}_t}{p}\cdot B(t,x)\nabla \hat q_{t} (x) \dd x\\
            =&- \int \nabla \log\frac{\hat{q}_t}{p} \cdot B(t_k,x) \nabla \log\frac{\hat q_t}{p} \hat q_{t} (x) \dd x
            \end{aligned}
    \end{equation}
    For the remaining terms, using Young's inequality in the form $\langle p,q\rangle = \langle (\epsilon B)^{1/2} p,(\epsilon B)^{-1/2} q\rangle\leq \frac{\epsilon}{2}|p|_B^2 + \frac{1}{2\epsilon}|q|_{B^{-1}}^2\leq \frac{\epsilon}{2}|p|_B^2 + \frac{1}{2\epsilon\minB}|q|^2$ and Jensen's inequality we find
    \begin{equation}
        \begin{aligned}
            C+D=
            &- \int \nabla \log\frac{\hat{q}_t}{p} \cdot \int (B(t_k,y) \nabla \Pot (y)-B(t_k,x) \nabla \Pot (x)) \hat q (y,t_k|x,t) \hat q_{t} (x) \dd y \dd x\\
            &- \int \nabla \log\frac{\hat{q}_t}{p} \cdot \int \bigg\{(B(t_k,y)-B(t,x))\nabla \log\hat q (x,t|y,t_k)\bigg\} \hat q (y,t_k|x,t) \hat q_{t} (x) \dd y \dd x\\
            \leq& \int \epsilon |\nabla \log \big(\frac{\hat{q}_t}{p}\big)|^2_{B(t_k,x)} \hat q_t \dd x\\
            &+ \int \frac{1}{2\epsilon\minB} \bigg(\int (B(t_k,y) \nabla \Pot (y)-B(t_k,x) \nabla \Pot (x)) \hat q (y,t_k|x,t) \dd y\bigg)^2 \hat q_{t} (x) \dd x\\
            &+ \int \frac{1}{2\epsilon\minB} \bigg(\int \bigg\{(B(t_k,y)-B(t,x))\nabla \log\hat q (x,t|y,t_k)\bigg\} \hat q (y,t_k|x,t) \dd y \bigg)^2  \hat q_{t} (x) \dd x
        \end{aligned}
    \end{equation}
    By the moment bounds in~\cref{lemma:moment_bounds} the second term is bounded uniformly over $\bstep$, $k$. For the third term we estimate using the local Lipschitz continuity
    \begin{equation}
        \begin{aligned}
            \int &\bigg(\int \bigg\{(B(t_k,y)-B(t,x))\nabla \log\hat q (x,t|y,t_k)\bigg\} \hat q (y,t_k|x,t) \dd y \bigg)^2  \hat q_{t} (x) \\
            \leq& \iint \bigg\{\lipB (1+|x|^\nB + |y|^\nB)(|t-t_k|^\holderB + |x-y|)\\
            &\times\big|(2(t-t_k)B(t_k,y))^{-1}(x-y - (t-t_k) b_\step(t_k,y))\big|\bigg\}^2 \hat q (x,t,y,t_k)\dd x\dd y\\
            \leq& \frac{\lipB^2}{4\minB^2}\iint \bigg\{(1+|x|^\nB + |y|^\nB)(|t-t_k|^\holderB + |x-y|)\big|\frac{x-y}{t-t_k} - b_\step(t_k,y))\big|\bigg\}^2 \hat q (x,t,y,t_k)\dd x\dd y\\
            \leq& \frac{\lipB^2}{\minB^2}\iint \bigg\{(1+|x|^\nB + |y|^\nB)^2(|t-t_k|^{2\holderB} + |x-y|^2)\big|\frac{|x-y|^2}{|t-t_k|^2} + |b_\step(t_k,y)|^2)\big|\bigg\}\hat q (x,t,y,t_k)\dd x\dd y.
        \end{aligned}
    \end{equation}
    Using the tower law, we may write the above integral in an iterated way as $\iint \dots \hat q (x,t,y,t_k)\dd x\dd y = \int \int \dots q (x,t|y,t_k)\dd x \hat q_{t_k}(y)\dd y$. Moreover, we can write $x\sim \hat q (x,t|y,t_k)$ as $y + (t-t_k)\drift_\step(t_k,y) + \sqrt{2(t-t_k) B(t_k,y)}z$ with $z\sim\Nc(0,\id)$ so that $x-y = (t-t_k)\drift_\step(t_k,y) +\sqrt{2(t-t_k) B(t_k,y)}z$. In particular, we have for any $p\geq 1$
    \begin{equation}
        |x-y|^p\leq 2^{p-1}|t-t_k|^p|\drift_\step(t_k,y)|^p + 2^{3p/2-1}|t-t_k|^{p/2}\maxB^{p/2}|z|^p
    \end{equation}
    thus
    \begin{equation}
        \int (1+|x|^\nB + |y|^\nB)^2 |x-y|^p\hat q(x,t,y,t_k)\dd x \dd y
        \leq C |t-t_k|^{p/2}
    \end{equation}
    for some constant $C>0$ due to the moment bounds.
    In particular, for $\delta\geq 1/2$
    \begin{equation}\label{eq:bounded_kl1}
        \begin{aligned}
            \int |t-t_k|^{2(\delta-1)}|y-x|^2 \hat q(x,t|y,t_k)\dd x\leq C_1<&\infty\\
            \int |t-t_k|^{-2}|y-x|^4 \hat q(x,t|y,t_k)\dd x\leq C_2 <&\infty
        \end{aligned}
    \end{equation}
    where we emphasize both constants $C_1$, $C_2$ in~\eqref{eq:bounded_kl1} remain bounded as $t-t_k\rightarrow 0$ and can, thus, be chosen independently of $\bstep$. In summary, we find 
    \begin{equation}
        \begin{aligned}
            C+D
            \leq& \int \epsilon |\nabla \log \big(\frac{\hat{q}_t}{p}\big)|^2_{B(t_k,x)} \hat q_t \dd x + \frac{1}{\epsilon}C_3
        \end{aligned}
    \end{equation}
    Combining this with the estimate for $A+B$ and choosing $\epsilon = 1/2$ we obtain 
    \begin{equation}
        \begin{aligned}
            \frac{\dd}{\dd t} \KL(\hat \mu_t|\pi)
            \leq& - \frac{1}{2}\int \nabla \log\frac{\hat{q}_t}{p} \cdot B(t_k,x) \nabla \log\frac{\hat q_t}{p} \hat q_{t} (x) \dd x +2C_3\\
            \leq& - \frac{2}{\Cplsi(t_k)}\KL(\hat \mu_t|\pi) +2C_3\\
            \leq& - \frac{2}{\Cplsi^{\max}}\KL(\hat \mu_t|\pi) +2C_3
        \end{aligned}
    \end{equation}
    Thus, we have 
    \begin{equation}
        \frac{\dd}{\dd t}(\KL(\hat \mu_t|\pi) - \Cplsi^{\max}C_3)\leq - \frac{2}{\Cplsimin}(\KL(\hat \mu_t|\pi)-\Cplsi^{\max}C_3)
    \end{equation}
    Gronwall's lemma then yields
    \begin{equation}
        \KL(\hat \mu_{k+1}|\pi) - \Cplsi^{\max}C_3\leq \expo{-\frac{2\step_k}{\Cplsi^{\max}}}(\KL(\hat \mu_{k}|\pi)- \Cplsi^{\max}C_3).
    \end{equation}
    As a consequence, whenever $\KL(\hat \mu_{k}|\pi)\leq \Cplsi^{\max}C_3$, then this remains the case for all. Conversely, whenever $\KL(\hat \mu_{k}|\pi)\geq  \Cplsi^{\max}C_3$, then $\KL(\hat \mu_{k}|\pi)\geq \KL(\hat \mu_{k+1}|\pi)$ which concludes the proof.
\end{proof}

\subsection{Auxiliary results}\label{sec:appendix auxiliary}
In this section we collect several continuity and growth results regarding drift and diffusion of the preconditioned \gls{sde} as well as its tamed discretization.

\begin{lemma}\label{lemma:drift_locally_lip}
    Under~\cref{ass}, $b:\R^d\rightarrow\R$ is locally Lipschitz in space and locally Hölder in time. More precisely, there exist $\lipdrift>0$, such that
    \begin{equation}
        |b(t,x) - b(s,y)|\leq \lipdrift (1+|x|^{\nB+1} + |y|^{\nB+1})(|x-y| + |s-t|^\holderB)
    \end{equation}
\end{lemma}
\begin{proof}
    Local Lipschitz continuity of $(t,x)\mapsto \divergence B(t,x)$ follows directly by \cref{ass} as we have
    \begin{equation}\label{eq:lipdrift1}
        \begin{aligned}
            |\divergence B(t,x) - \divergence B(s,y)|
            &\leq |\divergence B(t,x) - \divergence B(s,x)| + |\divergence B(s,x) - \divergence B(s,y)|\\
            &\leq \lipB (1+|x|^\nB)|s-t|^\holderB + \lipB (1+|x|^\nB + |y|^\nB)|x-y|.
        \end{aligned}
    \end{equation}
    Regarding the local Lipschitz continuity of $x\mapsto B(t,x) \nabla \Pot (x)$, since $B$ exhibits the same continuity properties as $\divergence B$ we simply compute
    \begin{equation}
        \begin{aligned}
            |B(t,x)\nabla\Pot(x)-&B(s,y)\nabla\Pot(y)|\\
            \leq& |B(t,x)-B(s,y)||\nabla\Pot(x)| + |B(s,y)||\nabla\Pot(x)-\nabla\Pot(y)|\\
            \leq& \sup_{|z|\leq |x|}|\nabla\Pot(z)| (\lipB (1+|x|^\nB)|s-t|^\holderB + \lipB (1+|x|^\nB + |y|^\nB)|x-y|) \\
            &+ \sup_{|z|\leq |y|}|B(s,z)| \lip|x-y|.
        \end{aligned}
    \end{equation}
    By Lipschitz continuity of $\nabla\Pot$ we have $|\nabla\Pot(z)| \leq \lip |z| + |\nabla \Pot(0)|$ which, together with $B\preceq\maxB\id$ implies the desired result.
\end{proof}

\begin{lemma}
    The mapping $B^{1/2}(t,x)$ is locally Hölder in time and locally Lipschitz in space, \ie, 
    \begin{equation}
        \begin{aligned}
            |B^{1/2}(t,x)-B^{1/2}(t,y)|\vee &\leq \frac{\lipB}{2\sqrt{\minB}} (1+|x|^\nB + |y|^\nB) |x-y|\\
            |B^{1/2}(s,x)-B^{1/2}(t,x)|\vee &\leq \frac{\lipB}{2\sqrt{\minB}} (1+|x|^\nB) |s-t|^\holderB.
        \end{aligned}
    \end{equation}
\end{lemma}
\begin{proof}
    We have for any $z\in\R^d\setminus\{0\}$
    \begin{equation}
        \begin{aligned}
            |z^\t &B^{1/2}(t,x) z - z^\t B^{1/2}(s,y) z| \\
            =& \frac{|z^\t B^{1/2}(t,x) z - z^\t B^{1/2}(s,y) z|\;(z^\t B^{1/2}(t,x) z + z^\t B^{1/2}(s,y) z)}{z^\t B^{1/2}(t,x) z + z^\t B^{1/2}(s,y) z}\\
            =& \frac{||z|^2z^\t B(t,x) z  -(z^\t B^{1/2}(s,y)z) (z^\t B^{1/2}(t,x)z) + (z^\t B^{1/2}(t,x)z)(z^\t B^{1/2}(s,y)z) - |z|^2z^\t B(s,y) z|}{z^\t B^{1/2}(t,x) z + z^\t B^{1/2}(s,y) z}\\
            =& \frac{||z|^2z^\t B(t,x) z - |z|^2z^\t B(s,y) z|}{z^\t B^{1/2}(t,x) z + z^\t B^{1/2}(s,y) z}\\
            \leq & \frac{|z|^2|z^\t (B(t,x)-B(s,y)) z|}{2\minB^{1/2}|z|^2}\\
            \leq & \frac{1}{2}\minB^{-1/2}|z|^2 |B(t,x)-B(s,y)|
        \end{aligned}
    \end{equation}
    implying $|B^{1/2}(t,x)-B^{1/2}(s,y)|\leq \frac{1}{2}\minB^{-1/2} |B(t,x)-B(s,y)|$ so that the assertion directly follows from~\cref{ass}.
\end{proof}

\begin{lemma}\label{lemma_approx_taming}
    There exists $\tamingconst>0$ such that
    \begin{equation}
        |b_\step(t,x) - b(t,x)|\leq \step \tamingconst (1+|x|^{\nB+1}).
    \end{equation}
\end{lemma}
\begin{proof}
    The proof is a simple computation
    \begin{equation}
        \begin{aligned}
            |b_\step(t,x) - b(t,x)| 
            =& \frac{\step |b(t,x)|}{1+\step |b(t,x)|}
            \leq \step |b(t,x)|
        \end{aligned}
    \end{equation}
    implying the result by boundedness of $B$, linear growth of $\nabla\Pot$, and~\cref{ass}.
\end{proof}

\begin{lemma}\label{lemma_taming_bound}
    There exists $\tamebound>0$ such that 
    \begin{equation}
        |b_\step(t,x)|\leq\frac{1}{\step}\wedge (1+\step)\tamingconst (1+|x|^{\nB+1}).
    \end{equation}
\end{lemma}
\begin{proof}
    The first bound follows from the fact that
    \[
        \sup_{s\geq 0}\frac{s}{1+\step s} = \frac{1}{\step}.
    \]
    For the second bound we note that 
    \[
        |b_\step(t,x)|\leq |b_\step(t,x)-b(t,x)| + |b(t,x)|\leq (1+\step)|b(t,x)|
    \]
    implying the result as in~\cref{lemma_approx_taming}.
\end{proof}

\begin{lemma}\label{lemma:growth_tamed}
    There exists $\step_{\max}>0$ such that 
    \begin{equation}
        \lim_{r\rightarrow\infty}\sup_{\substack{|x|\geq r,\\t\geq 0,\, \step\leq\step_{\max}}} 2 \bigg\langle \frac{x}{|x|},b_\step(t,x)\bigg\rangle + \frac{\step}{|x|}|b_\step (t,,x)|^2<0.
    \end{equation}
\end{lemma}
\begin{rmk}
    We want to point out that it is integral that the above $\liminf$ is uniformly over $\step$.\footnote{Note that \cite[A.2]{brosse2019tamed} incorrectly state the condition only for each fixed step size despite the fact that the proofs rely on the constants $\kappa,M_1$ in the proof of \cite[Proposition 3]{brosse2019tamed} be chosen uniformly over the step size. However, \cite{brosse2019tamed} is, nonetheless, correct as their proof of \cite[Lemma 2]{brosse2019tamed} is strong enough to support also the uniform $\liminf$.}
\end{rmk}
\begin{proof}
    By~\cref{ass} there exist $\epsilon,r>0$ such that for $|x|>r$, $\langle \frac{x}{|x|},b(t,x)\rangle\leq -\epsilon |b(t,x)|$. It follows
    \begin{equation}
        \begin{aligned}
            2 \bigg\langle \frac{x}{|x|},&b_\step(t,x)\bigg\rangle + \frac{\step}{|x|}|b_\step (t,,x)|^2\\
            \leq& -2\epsilon \frac{|b(t,x)|}{1+\step|b(t,x)|} + \frac{\step}{|x|}\frac{|b(t,x)|^2}{(1+\step|b(t,x)|)^2}\\
            = & \frac{|b(t,x)|}{1+\step|b(t,x)|}\bigg(-2\epsilon  + \frac{\step}{|x|}\frac{|b(t,x)|}{1+\step|b(t,x)|}\bigg)\\
            = & \frac{|b(t,x)|}{1+\step|b(t,x)|}\bigg(-2\epsilon  + \frac{1}{|x|}\frac{|b(t,x)|}{\frac{1}{\step}+|b(t,x)|}\bigg)\\
            \leq & \frac{|b(t,x)|}{1+\step|b(t,x)|}\bigg(-2\epsilon  + \frac{1}{|x|}\frac{|b(t,x)|}{\frac{1}{\step_{\max}}+|b(t,x)|}\bigg).
        \end{aligned}
    \end{equation}
    We can choose $r$ sufficiently large independently of $\step$ so that for $|x|\geq r$
    \begin{equation}
        \frac{1}{|x|}\frac{|b(t,x)|}{\frac{1}{\step_{\max}}+|b(t,x)|}<\epsilon.
    \end{equation}
    It follows 
    \begin{equation}
        \begin{aligned}
            2 \bigg\langle \frac{x}{|x|},&b_\step(t,x)\bigg\rangle + \frac{\step^2}{|x|}|b_\step (t,,x)|^2
            \leq -\epsilon \frac{|b(t,x)|}{1+\step_{\max}|b(t,x)|}.
        \end{aligned}
    \end{equation}
    Lastly, by again potentially increasing $r$ so that $|b(t,x)|>0$ for $|x|\geq r$ we can ensure that
    \begin{equation}
        \begin{aligned}
            \inf_{\substack{|x|\geq r,\\t\geq 0,\, \step\leq\step_{\max}}}\frac{|b(t,x)|}{1+\step_{\max}|b(t,x)|}>0
        \end{aligned}
    \end{equation}
    due to the facts $|b(t,x)|\rightarrow\infty$ as $|x|\rightarrow \infty$ uniformly in $t$ and since the function $s\mapsto s/(1+s\step_{\max})$ monotonically increasing for $s\geq 0$.
\end{proof}

\begin{lemma}\label{lemma:novikov}
    Let $(Y_t)_{t=0}^T$ be an adapted process, $\alpha>0$ and assume that Novikov's condition is satisfied, that is,
    \begin{equation}
        \E\left[ \expo{\frac{\alpha^2}{2}\int_0^T |Y_s|^2 \dd s } \right]<\infty.
    \end{equation}
    Then it holds true that
    \begin{equation}
        \P[\sup_{0\leq t\leq T} Y_t-\frac{\alpha}{2}\langle Y\rangle_t\geq R]\leq \expo{-\alpha R}.
    \end{equation}
\end{lemma}
\begin{proof}
    By the assumption that Novikov's condition is satisfied we have that
    \begin{equation}
        M_t = \expo{\int_0^t \alpha Y_s \dd W_s - \frac{\alpha^2}{2}\int_0^t |Y_s|^2\dd s}
    \end{equation}
    is a martingale with $\E[|M_t|] = \E[M_t] = 1$ for all $t$ (\cf~\cite[Corollary 5.13, Section 3.5]{karatzas1991brownian}). By the submartingale inequality (\cf~\cite[Theorem 3.8, Section 1.3]{karatzas1991brownian}) it follows that
    \begin{equation}
        \begin{aligned}
            \P[\sup_{0\leq t\leq T}Y_t-\frac{\alpha}{2}\langle Y\rangle_t\geq R] 
            = \P[\sup_{0\leq t\leq T}M_t\geq \expo{\alpha R}] 
            \leq \frac{\E[M_t]}{\expo{\alpha R}} = \expo{-\alpha R}.
        \end{aligned}
    \end{equation}
\end{proof}

\subsection{Details for the two-dimensional Rosenbrock potential}
\label{appendix:details_rosenbrock}
To produce ground truth data $X = (X_1, X_2) \sim \pi$ from the Rosenbrock target, we use ancestral sampling, \ie, we perform
\begin{equation}
    \begin{cases}
        X_1 \sim \Nc (a, \frac{1}{2}),\\
        X_2 \sim \Nc(X_1^2, \frac{1}{2b}).
    \end{cases}
\end{equation}
The Rosenbrock \gls{pdf} is given as
\begin{equation}
    p(x) = \frac{1}{Z} \exp(- \Pot(x)),
\end{equation}
with partition function $Z = \frac{\pi}{\sqrt{b}}$.
Further, the gradient $\nabla \Pot$ and Hessian $\nabla^2 \Pot$ of the Rosenbrock potential are implemented numerically in PyTorch in a batch-based fashion using automatic differentiation.

\subsection{Details for Bayesian logistic regression}
\label{appendix:details_blr}
The posterior distribution is given as 
 \begin{equation*}
     p(\beta \mid \{x_i, y_i\}_{i=1}^N) \propto \prod_{i=1}^N \varphi(\beta^\top x_i)^{y_i}(1 - \varphi(\beta^\top x_i))^{1 - y_i}p(\beta).
 \end{equation*}
The potential is derived as
\begin{align*}
     \Psi(\beta \mid \{x_i, y_i\}_{i=1}^N) &= -\sum_{i=1}^N y_i \log \varphi(\beta^\top x_i) + (1 - y_i)\log (1 - \varphi(\beta^\top x_i)) - \log p(\beta) \\
     &= -\sum_{i=1}^N y_i (\beta^\top x_i - \log(1 + \exp(\beta^\top x_i))) + \log (1 - \varphi(\beta^\top x_i)) \\
     &\hspace{1.5cm}- y_i\log (1 - \varphi(\beta^\top x_i)) - \log p(\beta) \\
     \intertext{by noting that \(\log (1 - \varphi(u)) = - \log (1 + \exp(u))\,,  u \in \mathbb{R}\) we obtain}
     &= \sum_{i=1}^N \log(1 + \exp(\beta^\top x_i)) - y_i \beta^\top x_i - \log p(\beta)
     \intertext{plugging in the prior introduced in \Cref{sec:bayes_log_ref} yields}
     &= \sum_{i=1}^N \log(1 + \exp(\beta^\top x_i)) - y_i \beta^\top x_i + \frac{1}{2} \|\beta\|_{\Sigma^{-1}}^2.
 \end{align*}
The gradient of \(\Psi\) is
\begin{align*}
    \nabla \Psi(\beta \mid \{x_i, y_i\}_{i=1}^N) &= \sum_{i=1}^N \frac{x_i \exp(\beta^\top x_i)}{1 + \exp(\beta^\top x_i)} - y_ix_i + \Sigma^{-1} \beta \\
    &= \sum_{i=1}^N x_i(\varphi(\beta^\top x_i) - y_i) + \Sigma^{-1} \beta \\
    &= X^\top(\varphi(X\beta) - y) + \Sigma^{-1}\beta.
\end{align*}
The Hessian is found to be 
\begin{equation*}
    \nabla^2 \Psi(\beta \mid \{x_i, y_i\}_{i=1}^N) = X^\top \diag(s_1, \dots, s_N) X + \Sigma^{-1},
\end{equation*}

where \(s_i = \varphi(\beta^\top x_i)(1 - \varphi(\beta^\top x_i))\).
Recalling that \(\Sigma = \diag(\sigma_1^2, \dots, \sigma_d^2)\) and using the bound $0 < \varphi(u)(1 - \varphi(u)) \leq \frac14$ for all $u \in \mathbb{R}$, we get that $L \leq \frac{1}{4}\lambda_{\max}(\bar X^\top\bar X) + \max_i(1/ \sigma_i^2)$. Moreover, since \(X^\top \diag(s_1, \dots, s_N) X \succeq 0\) and \(\sigma_1^2, \dots \sigma_d^2 > 0\) it follows that \(\nabla^2 \Psi \succ 0\) (\emph{i.e.}, all Eigenvalues of the Hessian are strictly positive). Therefore, we omit Eigenvalue clamping (see \Cref{sec:numerical}) when calculating the curvature-aware preconditioner in the Bayesian logistic regression experiments.

\subsection{Additional results for the two-dimensional Rosenbrock potential}
In~\cref{fig:rosenbrock-stats,fig:rosenbrock-final-samples,fig:rosenbrock-acf} we show additional numerical results for the Rosenbrock potential.
\begin{figure}
    \centering
    \includegraphics[]{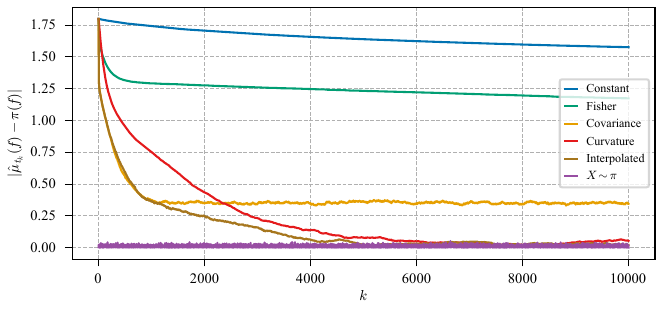}\\
    \includegraphics[width=1.0\linewidth]{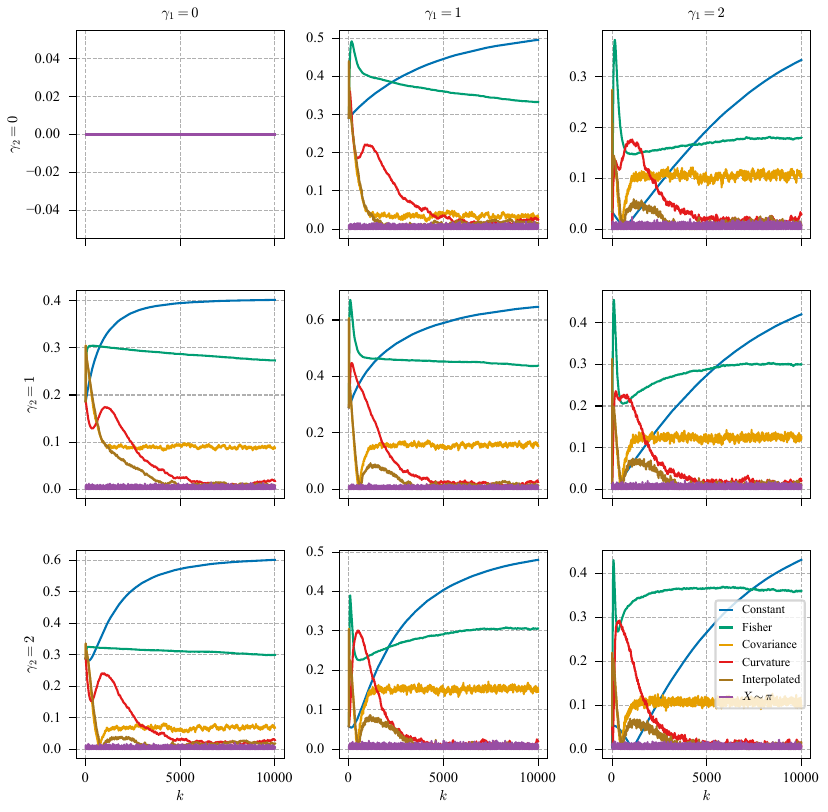}
    \caption{The evolution of the estimation error for various statistics over $K = \num{1e4}$ Langevin iterations with discretization step size $h= \num{6e-3}$ on the Rosenbrock distribution.
    \emph{Top}: Estimation of the mean $f(x) =x$.
    \emph{Bottom}: Estimation error of cosine waves $|\hat \mu_{t_k}(f_{\gamma_1, \gamma_2}) - \pi(f_{\gamma_1, \gamma_2})|$ with $f_{\gamma_1, \gamma_2}(x) = \cos(\gamma_1 x_1 + \gamma_2 x_2)$ for $(\gamma_1, \gamma_2) \in \{0,1,2\}^2$.
    For both tasks, we observe that \emph{Curvature} and \emph{Interpolated} are able to faithfully estimate the statistic of interest. On the other hand, the non-local preconditioning schemes merely provide biased estimates or even diverge.}
    \label{fig:rosenbrock-stats}
\end{figure}

\begin{figure}
    \centering
    \includegraphics[width=1.0\linewidth]{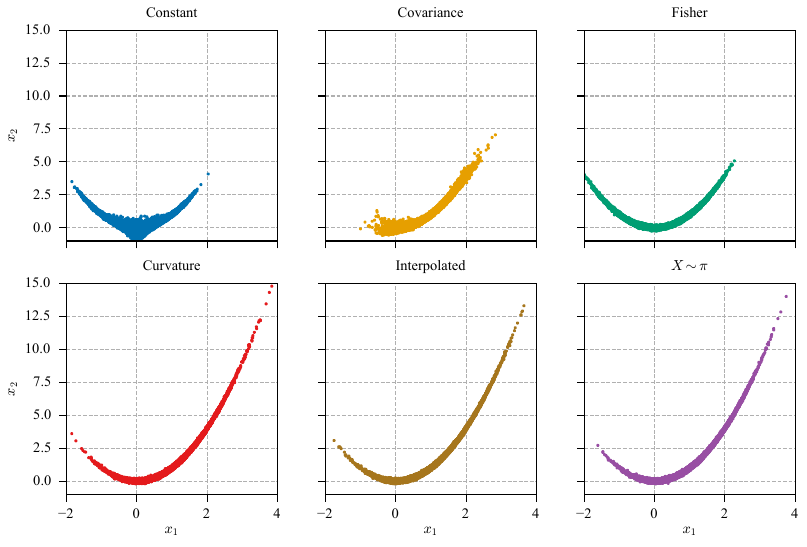}
    \caption{The final samples produced by the compared preconditioned dynamics.
    Evidently, both \emph{Curvature} and \emph{Interpolated} best match the ground truth samples depicted on the bottom right. All non-local preconditioning schemes struggle to capture the tails of the Rosenbrock distribution. The step size has been set to the optimal $h= \num{6e-3}$ and $K=\num{1e4}$ Langevin steps have been performed.}
    \label{fig:rosenbrock-final-samples}
\end{figure}

\begin{figure}
    \centering
    \includegraphics[width=1.0\linewidth]{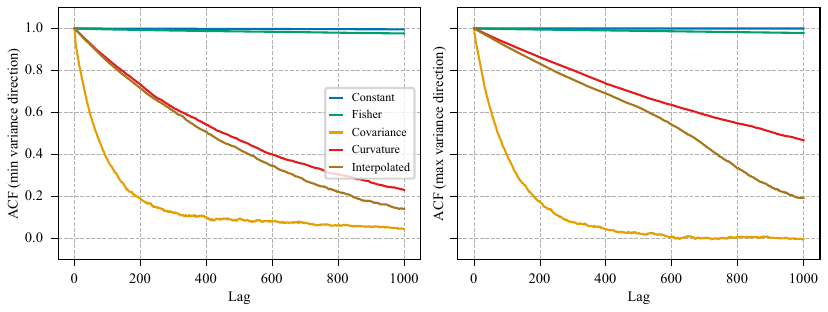}
    \caption{The observed \gls{acf} drop-off for increasing lag on the Rosenbrock distribution.
    The chains have been initialized in the target, and $K=\num{1e3}$ Langevin steps have been performed.
    Thereafter, an empirical covariance matrix was computed from the final samples, and all samples were projected onto the eigenvectors of this matrix.
    The Pearson correlation coefficient between the resulting statistics for different lags is depicted.
    It is observable that \emph{Covariance} preconditioning offers the quickest decrease in \gls{acf}.
    However, it should be noted that, at the same time, the \emph{Covariance} dynamics sample from a distribution with a strong bias toward the target.
    On the other hand, \emph{Curvature} and \emph{Interpolated} also offer faster \gls{acf} drop-off compared to \emph{Constant} and \emph{Fisher}, while maintaining a stationary distribution close to the actual target.}
    \label{fig:rosenbrock-acf}
\end{figure}

\subsection{Additional results for Bayesian logistic regression}
\label{appendix:add_results_blr}
In~\cref{fig:app_delta_init,fig:app_delta_init_2,fig:app_randn_inint,fig:sweep_randn} we show additional results of the experiment conducted in \Cref{sec:bayes_log_ref}.

\begin{figure}[!h]
    \centering
    \includegraphics[width=0.49\linewidth]{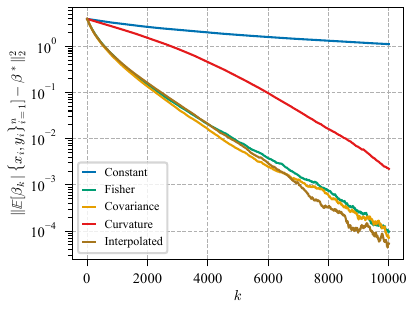}
    \includegraphics[width=0.49\linewidth]{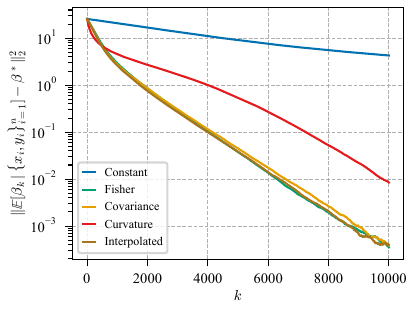}
    \caption{Performance of the investigated dynamics on the Bayesian logistic regression posterior under two initializations: \emph{left}\,\(\delta_0\); \emph{right}, \(\delta_{-1}\) for $h=\num{5e-4}$.} 
    \label{fig:app_delta_init}
\end{figure}

\begin{figure}[!h]
    \centering
    \includegraphics[width=0.49\linewidth]{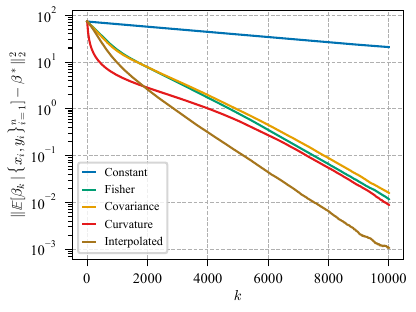}
    \includegraphics[width=0.49\linewidth]{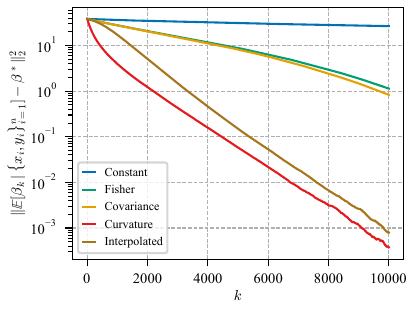}
    \caption{Performance of the investigated dynamics on the Bayesian logistic regression posterior under two initializations: \emph{left}\,\(\delta_{-2}\); \emph{right}, \(\delta_{2}\) for $h=\num{5e-4}$.} 
    \label{fig:app_delta_init_2}
\end{figure}

\begin{figure}[!h]
    \centering
    \includegraphics[width=0.49\linewidth]{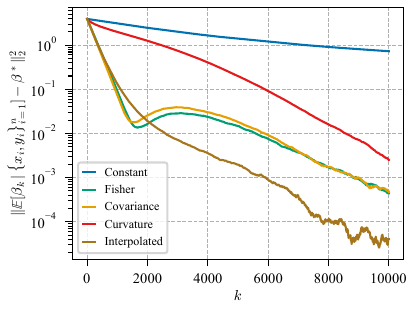}
    \includegraphics[width=0.49\linewidth]{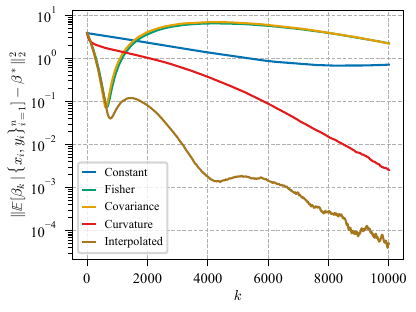}
    \caption{Performance of the investigated dynamics on the Bayesian logistic regression posterior under two initializations: \emph{left}\,\(\mathcal{N}(0,1\mathrm{I})\); \emph{right}, \(\mathcal{N}(0,3\mathrm{I})\) for $h=\num{5e-4}$.} 
    \label{fig:app_randn_inint}
\end{figure}

In \Cref{fig:app_delta_init} and \Cref{fig:app_randn_inint}, we observe that our method (\ie the interpolated preconditioner) exhibits at least as good of a convergence behavior as the global preconditioners (see \Cref{fig:app_delta_init}) for some initiliazations, while in other cases (see \Cref{fig:app_randn_inint} and \Cref{fig:app_delta_init_2}) our proposed method shows clear advantages and avoids slow converging paths. This advantage (see the right plot of \Cref{fig:app_randn_inint}) becomes especially pronounced, when the initialization is further away from the posterior mean (the maximal and minimal entries of \(\beta^*\) were estimated to be $1.2538$ and $-0.4581$).

\begin{figure}[!h]
    \centering
    \includegraphics[width=0.32\linewidth]{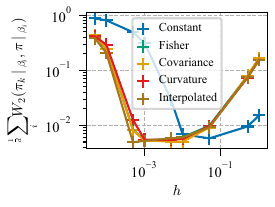}
    \includegraphics[width=0.32\linewidth]{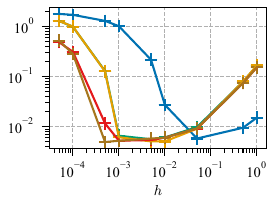}
    \includegraphics[width=0.32\linewidth]{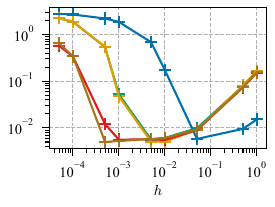}
    \caption{Step size sweep for a fixed budged of \(\num{1e4}\) steps. Initialization \emph{left}\,\(\mathcal{N}(0,1\mathrm{I})\); \emph{middle}, \(\mathcal{N}(0,2\mathrm{I})\) \emph{right}, \(\mathcal{N}(0,3\mathrm{I})\).} 
    \label{fig:sweep_randn}
\end{figure}

In \Cref{fig:sweep_randn}, we report the average marginal Wasserstein-2 distance for a log-linear sweep of step sizes \(h \in [\num{5e-5},1] \) across three Gaussian initializations. Overall, the preconditioned dynamics admit fast convergence for smaller step-sizes than the dynamics with constant, scalar preconditioning. While in this case, the best-case performance for the methods is similar, we would argue that fast convergence also for smaller step-sizes is overall a desirable property as it allows to approximate the continuous dynamics more accurately. Moreover, we believe that the similar convergence results for all methods are also a consequence of the Gaussian prior rendering the distribution significantly better behaved than the Rosenbrock potential, despite the ill-conditioned choice of the covariance.
An additional benefit of curvature-aware and interpolated preconditioners is that they exhibit similar behavior across methods for larger step sizes and initializations.

\end{document}